\documentclass[bj]{imsart}

\RequirePackage{amsthm,amsmath,amsfonts,amssymb}
\RequirePackage[numbers,sort&compress]{natbib}
\RequirePackage[colorlinks,citecolor=blue,urlcolor=blue]{hyperref}

\usepackage{mathrsfs}
\usepackage{dsfont}
\usepackage{color}
\usepackage{bbm}
\usepackage[latin1]{inputenc}
\usepackage{csquotes}
\usepackage{graphicx,epstopdf}
\usepackage[dvips]{epsfig}
\usepackage{latexsym}
\usepackage{exscale}
\usepackage[toc]{appendix}
\usepackage{diagbox}
\usepackage{multirow}
\usepackage{tcolorbox}
\usepackage{xcolor}
\usepackage{soul}
\usepackage{ulem}

\usepackage{tabularx}
\usepackage{array}
\usepackage{colortbl}
\usepackage{float}

\startlocaldefs
\theoremstyle{plain}
\newtheorem{theorem}{Theorem}[section]
\newtheorem{lemma}[theorem]{Lemma}

\theoremstyle{remark}
\newtheorem{remark}[theorem]{Remark}
\newtheorem{definition}[theorem]{Definition}
\def\ud{\, \mathrm{d}}

\numberwithin{equation}{section}
\endlocaldefs

\begin{document}
\begin{frontmatter}
\title{Linear Multifractional Stable Sheets in the Broad Sense: Existence and Joint Continuity of Local Times}
\runtitle{Linear Multifractional Stable Sheets in the Broad Sense}

\begin{aug}
\author[A]{\fnms{Yujia} \snm{Ding}\ead[label=e1,mark]{yujia.ding@cgu.edu}}
\author[A]{\fnms{Qidi} \snm{Peng}\ead[label=e2,mark]{qidi.peng@cgu.edu}}
\and
\author[B]{\fnms{Yimin} \snm{Xiao}
\ead[label=e3]{xiaoy@msu.edu}}
\address[A]{Institute of Mathematical Sciences, Claremont Graduate University,
\printead{e1,e2}}

\address[B]{Department of Statistics and Probability, Michigan State University,
\printead{e3}}

\runauthor{Y. Ding et al.}
\end{aug}

\begin{abstract}
We introduce the notion of \textit{linear multifractional stable sheets in the broad sense} (LMSS) with $\alpha\in(0,2]$, 
to include both linear multifractional Brownian sheets ($\alpha=2$) and linear multifractional stable sheets ($\alpha<2$). 
The purpose of the present paper is to study the existence and joint continuity of the local times of LMSS, 
and also the local H\"older condition of the local times in the set variable. Among the main results of this 
paper,  Theorem \ref{thm:existence} provides a \textit{sufficient and necessary condition} for the existence of local times 
of LMSS; Theorem \ref{thm::joint_con} shows a  \textit{sufficient condition} for the joint continuity of local times; and  
Theorem \ref{thm::holder}  proves a sharp local H\"older condition for the local times in the set variable. All these 
theorems improve significantly the existing results for the local times of multifractional Brownian sheets and linear 
multifractional stable sheets in the literature.
\end{abstract}

\begin{keyword}[class=MSC2020]
\kwd[Primary ]{60G52}
\kwd{60G22}
\kwd[; secondary ]{60G60}
\end{keyword}

\begin{keyword}
\kwd{Linear multifractional Brownian sheets}
\kwd{linear multifractional stable sheets}
\kwd{local times}
\kwd{joint continuity}
\end{keyword}
\end{frontmatter}


\section{Introduction}
\label{sec:introduction}
The purpose of this paper is to develop a unified framework for improving the results concerning the local times of multifractional 
Brownian sheets \cite{MWX08}, linear fractional stable sheets \cite{ARX07b}, and linear multifractional stable sheets 
\cite{shen2020local}. We will use the notion of linear multifractional stable sheets in the broad sense (LMSS), 
where the stability index is  $\alpha$, which controls the tail-heaviness of the distributions, is ranged in $(0,2]$. 
As a consequence of the present paper, some results obtained in \cite{ARX07b,AWX08,AX05,BDG06,MWX08,shen2020local,X97,X11} 
are extended to the setting of LMSS and improved significantly. Below we describe the main contributions of this paper:
\begin{description}
\item[(i).] A sufficient condition for the existence of local times of multifractional Brownian sheets was given in \cite{MWX08}, 
which was extended to linear multifractional stable sheets  in \cite{shen2020local}. However, their conditions are not optimal. 
In particular, no \textit{necessary condition} had been proved for the existence of local times of multifractional 
Brownian sheets or linear multifractional stable sheets in the literature. We fill this gap by proving in Theorem \ref{thm:existence} a sufficient 
and necessary condition for the existence of the local times of LMSS. This solves completely the problem on the existence of local times
for multifractional Brownian sheets in \cite{MWX08} and linear multifractional stable sheets in \cite{shen2020local}.
\item[(ii).] In Theorem \ref{thm::joint_con}, we provide a sufficient condition for the joint continuity of the local times of LMSS, which 
is significantly weaker than the conditions proved in \cite{MWX08}  for multifractional Brownian sheets  and in 
\cite{shen2020local} for linear multifractional stable sheets. We remark that \cite{shen2020local} makes crucial use of 
the arguments in \cite{X11}, which rely on the local nondeterminism property and the assumption of 
$\alpha \in (1, 2)$. Our Theorem \ref{thm::joint_con} holds for all  $\alpha \in (0,2]$ and its proof builds upon 
an extension of the direct approach in  \cite{ARX07a} for linear fractional stable sheets that can provide more precise 
information on the upper bound for the moments of local times than those in \cite{shen2020local,X11}. 
\item[(iii).] We prove a local H\"older condition for the  local times of LMSS, see Theorem \ref{thm::holder}. 
This theorem is useful for studying the  local Hausdorff dimension and exact Hausdorff measure of the level sets of 
LMSS \cite{AX05,X99}. This latter problem goes beyond the scope of the present paper and we plan to study it in a 
subsequent paper. 
\item[(iv).] Through proving the aforementioned theorems we have extended and improved several results in the literature. 
These include Lemma \ref{int_exp}, Lemma \ref{lem::bound_increm} and Remark \ref{cor_Lemma}, which 
may have their own interests.
\end{description}

Throughout this paper, if not specified, we adopt the following notations and assumptions:
\begin{itemize}
\item For $0<\epsilon<T$, let $I:=[\epsilon,T]^N$. For every $u,~v\in\mathbb{R}^N$ such that $u_l \le v_l~(l=1,\ldots,N)$, 
$[u, v]$ denotes the closed rectangle defined by $[u,v]:=\prod_{l=1}^N [u_l, v_l]$. 
\item For $x,y\in\mathbb R$, define $x\vee y:=\max\{x,y\}$ and $x\wedge y:=\min\{x,y\}$.
\item Denote by $\left|\bullet\right|$ the Euclidean norm.
\end{itemize}

Before we introduce LMSS, let us first define the linear fractional stable
sheet in the broad sense (LFSS). The phrase ``broad sense" refers to the fact that we allow $\alpha\in(0,2]$, instead of treating
the two cases of  $\alpha= 2$ and $\alpha \in(0,2)$ separately.
\begin{definition}[Linear fractional stable sheet in the broad sense]
\label{LFSS}
For any $\alpha\in(0,2]$, any vectorial exponent $H=(h_1,\ldots,h_N) \in (0,1)^N$, a real-valued linear fractional stable sheet in 
the broad sense  $X_0^H=\{X_0^H(u),u\in\mathbb{R}_{+}^N\}$ is defined via the following integral representation:
\begin{equation}\footnotesize
\label{eqn::def_lfss}
X_0^H(u) := \int_{\mathbb{R}^N} g^H(u,v) M_{\alpha}(\ud v), \quad \mbox{for all}~ u \in \mathbb{R}_{+}^N:=[0,+\infty)^N,
\end{equation}
where:
\begin{itemize}
\item For $\alpha \in (0, 2)$, $M_{\alpha}$ denotes a rotationally invariant $\alpha$-stable random measure on 
$\mathbb{R}^N$ with Lebesgue control measure. When $\alpha=2$, $M_\alpha$ stands for the standard Gaussian 
measure (or Gaussian white noise). See Samorodnitsky and Taqqu \cite{Taqqu1994} for the definition and properties 
of the integral in (\ref{eqn::def_lfss}).

\item The kernel function $g^H:~\mathbb R_+^N\times \mathbb R^N\to \mathbb R_+$ is defined as
\begin{equation}\footnotesize
\label{eqn::g_u_v}
g^H(u,v) := c_H\prod_{l=1}^N \left[(u_l-v_l)^{h_l-1/\alpha}_{+} - (-v_l)^{h_l-1/\alpha}_{+}\right],
\end{equation}
where $a_+:=\max\{a,0\}$ for $a \in \mathbb R$ and the normalizing constant $c_H>0$ is chosen such that 
$\|X_0^H(1)\|_\alpha=1$ (see the forthcoming (\ref{def:la_norm}) for the definition of $\|\bullet\|_\alpha$). 
\end{itemize}

For any integer $d\ge 1$, an $(N,d)$-LFSS is defined by
\footnotesize
$$
X^H=\big\{X^H(u),~u\in\mathbb R_+^N\big\}:=\Big\{\big(X_1^H(u),\ldots,X_d^H(u)\big),~u\in\mathbb R_+^N\Big\},
$$
\normalsize
where $X_1^H,\ldots,X_d^H$ are $d$ independent copies of $X_0^H$.
\end{definition}

\begin{remark}
When $\alpha=2$, $X^{H}$ is known as a fractional Brownian sheet. When $h_1=\ldots=h_N=1/\alpha$,  $X^H$ becomes 
the  ordinary stable sheet, studied in \cite{E81}.
\end{remark}

Several authors have studied the sample path properties of the LFSS $X^H$. For example, Ayache et al. \cite{ARX07a,ARX09} 
considered the local and asymptotic properties of the paths of the LFSS with $\alpha<2$. Ayache et al. \cite{ARX07b} proved 
the existence and joint continuity of the  local times of LFSS with $\alpha<2$, subject to some conditions on
$d$ and $H$. Xiao and Zhang \cite{Xiao2002} and Ayache et al. \cite{AWX08} studied the existence and joint continuity of the local
times of fractional Brownian sheets (LFSS with $\alpha=2$), respectively.  Xiao \cite{X11} proved that LFSS has the 
property of sectorial local nondeterminism and applied this property to study the regularity properties of the local times of LFSS.
The results show that the regularity and fractal properties of LFSS $X^H$ are determined by the stability index 
$\alpha$ and the constant vector $H=(h_1,\ldots,h_N).$ In particular, many random sets generated by $X^H$ such as its trajectories 
and level sets are ``monofractals".

In order to construct more flexible stochastic models with varying local regularity and fractal properties, 
Steov and Taqqu \cite{ST04,ST05} introduced linear multifractional stable motion and studied its stochastic 
and sample properties. Recently Shen et al. \cite{shen2020local} obtained a sufficient condition 
for the existence of local times of the linear multifractional stable sheets with $\alpha<2$.  By extending 
the definition of linear multifractional stable motion in \cite{ST04,ST05,shen2020local} 
to the random field setting in  Definition \ref{LFSS}, 
we define the linear multifractional stable sheet in the broad sense (LMSS) as follows.

\begin{definition}[Linear multifractional stable sheet in the broad sense]
\label{LMSS}
Let $\alpha\in (0,2]$ and let $H(u)=(h_1(u),\ldots,h_N(u))$, $u\in\mathbb{R}_+^N$, be a deterministic 
function such that
\footnotesize
\begin{equation*}
0<h_l(u)<1 \quad \textrm{ for all }u\in \mathbb{R_+^N}~\mbox{and all}~ l\in\{1,\ldots,N\}.
\end{equation*}
\normalsize
A real-valued linear multifractional stable sheet in the broad sense with Hurst functional index $H(\bullet)$ 
 is defined by
\begin{equation}\footnotesize
\label{eqn::def_lmss}
X_0^{H(u)}(u) := \int_{\mathbb{R^N}} g^{H(u)}(u,v) M_{\alpha}(\ud v), \quad u\in \mathbb{R}^N_{+},
\end{equation}
where $M_{\alpha}$  is given in Definition \ref{LFSS} and, for any $u, v \in\mathbb R_+^N$, 
 $g^{H(u)} (u, v)$ is defined as in (\ref{eqn::g_u_v}) with $h_l$ replaced by $h_l(u)$ and with its normalizing constant 
 $c_{H(1)}$ chosen to satisfy $\|X_0^{H(1)}(1)\|_\alpha=1$.

For $d\ge1$, an $(N,d)$-LMSS is defined to be
\footnotesize
\begin{equation*}
X^{H(\bullet)}=\big\{X^{H(u)}(u),~u\in\mathbb R_+^N\big\}:=\Big\{\big(X_1^{H(u)}(u),\ldots,X_d^{H(u)}(u) \big),~u\in\mathbb R_+^N\Big\},
\end{equation*}
\normalsize
where $X_1^{H(\bullet)},\ldots,X_d^{H(\bullet)}$ are $d$ independent copies of $X_0^{H(\bullet)}$.
\end{definition}
\begin{remark}\ \
\begin{itemize}
\item Recall from \cite{Taqqu1994,Nolan1989,X11} that, for $\alpha\in(0,2]$, if $\{\widetilde X(u),~u\in\mathbb R_+^N\}$
is a symmetric $\alpha$-stable random field having the following integral representation:
\begin{equation}\footnotesize
\label{def:X_general}
\widetilde X(u):=\int_{\mathbb R^N}g(u,v)M_\alpha(\ud v),~\mbox{for all}~u\in\mathbb R_+^N,
\end{equation}
then for all $a_1,\ldots,a_n\in\mathbb R,~u^1,\ldots,u^n\in\mathbb R_+^N$, the characteristic function of the joint distribution 
of $(\widetilde X(u^1),\ldots, \widetilde X(u^n))$
is given by
\begin{equation}\footnotesize
\label{characteristic}
\mathbb E\Big(e^{i\sum_{j=1}^na_j\widetilde X(u^j)}\Big)=e^{-\big\|\sum_{j=1}^na_j\widetilde X(u^j)
\big\|_\alpha^\alpha},
\end{equation}
where
\begin{equation}\footnotesize
 \label{def:la_norm}
    \Big\|\sum_{j=1}^na_j\widetilde X(u^j)\Big\|_\alpha^\alpha:=\int_{\mathbb R^N}\Big|\sum_{j=1}^na_jg(u^j,v)\Big|^\alpha\ud v.
\end{equation}
It is worth noting that $\|\bullet\|_\alpha$ defines an $L^\alpha$-norm only when $\alpha\ge1$; when $\alpha\in(0,1)$, $\|\bullet\|_\alpha$ does not satisfy the triangle inequality and it is called an $L^\alpha$-quasinorm.
\item In particular when $\widetilde X= X_0^{H(\bullet)}$, the characteristic function of $(X_0^{H(u^1)}(u^1)$, 
$\ldots,$ $X_0^{H(u^n)}(u^n))$
is given by
\footnotesize
\begin{equation}
\label{characteristic_LMSS}
\mathbb E\Big(e^{i\sum_{j=1}^na_jX_0^{H(u^j)}(u^j)}\Big)=e^{-\big\|\sum_{j=1}^na_jX_0^{H(u^j)}(u^j)
\big\|_\alpha^\alpha}
\end{equation}
\normalsize
for all $a_1,\ldots,a_n\in\mathbb R~$ and $~u^1,\ldots,u^n\in\mathbb R_+^N$, where 
\footnotesize
\begin{equation*}
    \Big\|\sum_{j=1}^na_jX_0^{H(u^j)}(u^j)\Big\|_\alpha^\alpha:=\int_{\mathbb R^N}\Big|\sum_{j=1}^na_jg^{H(u^j)}(u^j,v)\Big|^\alpha\ud v.
\end{equation*}
\normalsize
When $\alpha=2$, $X^{H(\bullet)}$ becomes a
multifractional Brownian sheet \cite{MWX08}. Since the normalizing factor in $g^{H(\bullet)}$ is chosen such that $\|X_0^{H(1)}(1)\|_2=1$,
we obtain that: for $\alpha=2$,
\footnotesize
\begin{equation*}
    \Big\|\sum_{j=1}^na_jX_0^{H(u^j)}(u^j)\Big\|_2^2=\frac{1}{2}{\mathbb Var}\bigg(\sum_{j=1}^na_jX_0^{H(u^j)}(u^j)\bigg).
\end{equation*}
\normalsize
 \item When $H(\bullet)\equiv H$ (constant), $X^{H(\bullet)}$ becomes the LFSS  with Hurst index $H$ in Definition \ref{LFSS} 
 (see \cite{ARX07b,ARX07a,ARX09}).
 \end{itemize}
\end{remark}
Next we recall the notion of local times as in \cite{GH80}. For more information on local times of Gaussian
and stable processes or random fields we refer the readers to
\cite{B70,Pitt1978,Nolan1989,D02,BDG06,ARX07b,MWX08,X11} and the references therein.

Let $Y: \mathbb R^N \to \mathbb R^d$ be a (deterministic or random) Borel vector field and let  $\lambda_N$ be
the Lebesgue measure on $\mathbb R^N$. For any Borel set $I\subseteq\mathbb{R}^N$,
the occupation measure of $Y$ on $I$ is the Borel measure on $\mathbb R^d$ defined   by
\begin{equation}\footnotesize \label{occupation}
\mu_I(\bullet) := \lambda_N\left\{t\in I: Y(t) \in \bullet\right\}.
\end{equation}
\begin{definition}[Local time]\label{Def:LT}
If $\mu_I$ is absolutely continuous with respect to the Lebesgue measure $\lambda_d$ in $\mathbb R^d$, 
$Y$ is said to have local times on $I$ and its local time $L(\bullet,I)$ is defined to be the Radon-Nikod\'ym 
derivative of $\mu_I$ with respect to $\lambda_d$, that is,
\footnotesize
\begin{equation*}
L(x,I) := \frac{\ud\mu_I}{\ud\lambda_d}(x), \quad \mbox{for a.e. }~ x \in \mathbb{R}^d,
\end{equation*}
\normalsize
where $x$ and $I$ are called the space variable and time variable, respectively.
\end{definition}
Heuristically speaking, for any Borel set $\mathcal A\subseteq \mathbb R^d$, $\mu_I(\mathcal A)$ measures 
the amount of \enquote{time} that $Y$ spends in $\mathcal A$ during the time period $I$;
$L(x,I)$ measures the amount of ``time'' that $Y$ spends at $x$ during $I$. Sometimes, we write $L(x,t)$ in 
place of $L(x,[0,t])$. 

The following remark contains two consequences of Definition \ref{Def:LT}. We will make use of the second 
observation in the proofs of the main results, Theorems \ref{thm:existence} 
and \ref{thm::joint_con}. 
\begin{remark}\label{re:LT}
Notice that if $Y$ has local times on $I$, for any Borel set $J\subseteq I$, $L(x,J)$ also exists.
On the other hand, if $I_1, \ldots,  I_n$ is an arbitrary partition of $I$ and $Y$ has local times on $I_i$ ($i = 1,\ldots, n$),
 $Y$ admits local times $L(x,I)$ on $I$ and $L(x,I) = \sum_{i=1}^n L(x,I_i)$ a.e..
\end{remark}

\begin{definition}[Joint continuity of local time]
Let $Y: \mathbb R^N \to \mathbb R^d$ be a random field and let  $I=[a,b]\subseteq \mathbb R^N$. If the local 
time of $Y$, $L(x,[a,t])$, is an almost surely continuous function of $(x,t) \in \mathbb{R}^d\times [a,b]$, we say that
$Y$ has a jointly continuous local time on $I$.
\end{definition}

In order to study the existence and joint continuity of the local times of LMSS, we introduce some assumptions 
on the smoothness of the Hurst functional index  $H(\bullet)$, and assume they hold throughout the rest of the paper. 
\begin{enumerate}
\item[$\mathcal H_1$:] For $l=1,\ldots,N$, there are constants $0<m_l< M_l<1$  such that $m_l \le h_l(u) \le M_l$
 for all $u\in\mathbb{R}_{+}^N$.
\item[$\mathcal H_2$:] Let $I\subseteq \mathbb R_+^N$ be a compact set.  There is a constant 
$c=c(I)>0$ such that for $l=1,\ldots,N$,
\footnotesize
\begin{equation*}
|h_l(u) - h_l(v)| \leq c \rho(u,v),~\mbox{for all}~u,v \in I,
\end{equation*}
\normalsize
\end{enumerate}
where $\rho(u,v)$ is the metric in $\mathbb{R}^N$ defined by
\footnotesize
\begin{equation*}
\rho(u,v) := \sum_{l=1}^N \min\left\{ |u_l - v_l|^{m_l},|u_l-v_l|^{M_l}\right\},~ \mbox{for all}~u,v \in \mathbb{R}^N,
\end{equation*}
\normalsize
where $(m_1,\ldots,m_N)$ and $(M_1, \ldots, M_N) \in (0,1)^N$ are given in $\mathcal H_1$.

Notice that Condition $\mathcal H_2$ implies that the Hurst functional index $H(\bullet)$ is continuous. Moreover, 
when $|u-v|\le 1$  we have $\rho(u,v) = \sum_{l=1}^N |u_l-v_l |^{M_l}$. Therefore, the metric $\rho$ coincides locally with the metric $\rho_K$ introduced in \cite{MWX08}. 
\section{Existence of the local times}
\label{sec:local_times}
This section is devoted to studying the existence of local times of  the $(N, d)$-LMSS
$\{X^{H(u)}(u)\}_{u\in\mathbb R_+^N}$. As the first main result, Theorem \ref{thm:existence} derives a necessary and sufficient condition for the existence of local times. The key idea to its proof is: first, observe that the existence of local times is equivalent to
\begin{equation}\footnotesize
\label{proof_step1}
\int_I\int_I \big\| X_1^{H(u)}(u) - X_1^{H(v)}(v) \big\|_{\alpha}^{-d}\ud u\ud v<+\infty.
\end{equation}
Next, by using the fact that $\| X_1^{H(u)}(u) - X_1^{H(v)}(v) \|_{\alpha}$ is compatible with $\sum_{l=1}^N|u_l-v_l|^{ h_l(u)}$ for all $u, v \in I$ (see Lemma \ref{lem::alpha_X_upp_low_bound} below), we obtain that (\ref{proof_step1}) is equivalent to
\begin{equation}\footnotesize
\label{proof_step2}
\int_I\int_I \Big(\sum_{l=1}^N|u_l-v_l|^{h_l(u)}\Big)^{-d}\ud u\ud v<+\infty.
\end{equation}
Finally, by applying Lemma \ref{int_equiv} below to an argument by induction on $N$, we derive the necessary and sufficient condition 
for (\ref{proof_step2}) to hold.

Before stating Theorem \ref{thm:existence}, we give some preliminary results that are useful in its proof. Recall the following elementary inequalities.
\begin{lemma}
\label{ineq:triagnle}
 For any $\alpha>0$,
$x_1,\ldots,x_N\in\mathbb R$,
\begin{equation}\footnotesize
\label{ineq_tri}
\bigg|\sum_{l=1}^Nx_l\bigg|^\alpha\le\left\{
\begin{array}{ll}
N^{\alpha-1}\sum\limits_{l=1}^N|x_l|^\alpha,&\mbox{if $\alpha\ge1$};\\
\sum\limits_{l=1}^N|x_l|^\alpha,&\mbox{if $0<\alpha< 1$}
\end{array}\right.= \big(N^{\alpha-1}\vee1\big)\sum_{l=1}^N|x_l|^\alpha.
\end{equation}
\end{lemma}
\begin{proof}
The inequalities in (\ref{ineq_tri}) hold for $\alpha\ge1$ thanks to Jensen's 
inequality; they hold for $\alpha\in(0,1)$
due to the triangle inequality $
|\sum_{l=1}^Nx_l|^\alpha\le (\sum_{l=1}^N|x_l|)^\alpha
$ and the fact that the mapping $x\mapsto (x+a)^\alpha-x^\alpha-a^\alpha$ with $a>0$ is 
decreasing over $x\in\mathbb R_+$. 
\end{proof}

The lemma below describes an approximation of the increments of $X_0^{H(\bullet)}$ 
in the $L^\alpha$-(quasi)norm.
\begin{lemma}
\label{lem::alpha_X_upp_low_bound}
Let $X_0^{H(\bullet)}$ be an $(N,1)$-LMSS and let $ 0<\epsilon < T$ be 
two constants such that $|T-\epsilon|$ is sufficiently small.
Then there exist constants $0<c_{2,1}\le c_{2,2}$ such that for all $u,v\in I=[\epsilon,T]^N$,
\footnotesize 
\begin{equation}
\label{eqn::X_H_u_bounded}
c_{2,1} \sum_{l=1}^N |u_l-v_l|^{ h_l(\hat{u})} 
\leq \big\| X_0^{H(u)}(u) - X_0^{H(v)}(v) \big\|_{\alpha}\leq c_{2,2} \sum_{l=1}^N |u_l - v_l|^{h_l(\hat{u})},
\end{equation}
\normalsize
for any $\hat{u}:=\left(\hat{u}_1,\ldots,\hat{u}_N\right)\in\prod_{l=1}^N[u_l \wedge v_l, u_l \vee v_l]$.
\end{lemma}
\begin{proof}
First,  it was proved in \cite[ Lemma 2.2]{MWX08} 
for $\alpha=2$ and in \cite[Lemma 3.2]{shen2020local} for $0< \alpha<2$, that for $|T-\epsilon|$ being sufficiently small, 
there exist constants $0<c_{2,3}\le c_{2,4}$ such that for all $u,v\in I=[\epsilon,T]^N$,
\footnotesize 
\begin{equation}
\label{eqn::X_H_u_bounded_1}
c_{2,3} \Big(\sum_{l=1}^N |u_l-v_l|^{\alpha h_l(\hat{u})}\Big)^{1/\alpha} 
\leq \big\| X_0^{H(u)}(u) - X_0^{H(v)}(v) \big\|_{\alpha}\leq c_{2,4} \Big(\sum_{l=1}^N |u_l - v_l|^{\alpha h_l(\hat{u})}\Big)^{1/\alpha},
\end{equation}
\normalsize
for any $\hat{u}:=\left(\hat{u}_1,\ldots,\hat{u}_N\right)\in\prod_{l=1}^N[u_l \wedge v_l, u_l \vee v_l]$. Next using Lemma \ref{ineq:triagnle} yields
\footnotesize 
\begin{equation}
\label{eqn::X_H_u_bounded_2}
(N^{\alpha-1}\vee1)^{-1/\alpha}\sum_{l=1}^N |u_l-v_l|^{ h_l(\hat{u})}\le\Big(\sum_{l=1}^N |u_l-v_l|^{\alpha h_l(\hat{u})}\Big)^{1/\alpha}\le \big(N^{1/\alpha-1}\vee1\big)\sum_{l=1}^N |u_l-v_l|^{ h_l(\hat{u})}.
\end{equation}
\normalsize
Finally (\ref{eqn::X_H_u_bounded}) follows from (\ref{eqn::X_H_u_bounded_1}) and (\ref{eqn::X_H_u_bounded_2}).
\end{proof}
The following lemma is an extension  of \cite[Lemma 8.6]{X09}.
\begin{lemma}
\label{int_equiv}
	Let $\alpha>0,\beta \ge 0$ and  $0\leq a<b$ be constants. Then for all $A> 0$ and $t_0\in[a,b]$,
	\begin{equation}\footnotesize
	\label{lem::ele_int_result_2}
	\int_{a}^{b}(A+|t-t_0|^\alpha)^{-\beta}\ud t\asymp
	\left\{\begin{array}{lll}
	A^{-(\beta-1/\alpha)}&\mbox{if }\alpha\beta>1,\\
	\log\big((1+(b-t_0)A^{-1/\alpha})(1+(t_0-a)A^{-1/\alpha})\big)&\mbox{if }\alpha\beta=1,\\
	1&\mbox{if }\alpha\beta<1.
	\end{array}\right.
	\end{equation}
Here and below, for two positive real-valued functions $f$ and $B$ defined on a set $D$,  
$f \asymp B$ means that there exist $c_{2,5},c_{2,6}>0$ such that
	$c_{2,5}B(x)\le f(x) \le c_{2,6}B(x)$ for all $x\in D$.
\end{lemma}
\begin{proof}
On one hand, from \cite[Lemma 8.6]{X09} we see that  for any given constants $\alpha>0, \beta\ge0,$
 \begin{equation}\footnotesize
    \label{lem::ele_int_result_1}
	\int_{0}^{b}(A+t^\alpha)^{-\beta}\ud t\asymp
	\left\{\begin{array}{lll}
	A^{-(\beta-1/\alpha)}&\mbox{if }\alpha\beta>1,\\
	\log(1+bA^{-1/\alpha})&\mbox{if }\alpha\beta=1,\\
	1&\mbox{if }\alpha\beta<1
	\end{array}\right.~~~~~\mbox{for all $A>0$.}
	\end{equation}

On the other hand, by the change of variable $u=t-t_0$, we obtain
\begin{equation}\footnotesize
\label{lem::ele_int_result_3}
\int_{a}^{b}(A+|t-t_0|^\alpha)^{-\beta}\ud t
=\int_{0}^{t_0-a}(A+u^\alpha)^{-\beta}\ud u+\int_{0}^{b-t_0}(A+u^\alpha)^{-\beta}\ud u.
\end{equation}
Since $t_0-a,~b-t_0\ge0$, applying (\ref{lem::ele_int_result_1}) to the right-hand side of  (\ref{lem::ele_int_result_3}) yields (\ref{lem::ele_int_result_2}).
\end{proof}
As our first main result, Theorem \ref{thm:existence} below  provides a \textit{sufficient and necessary condition} $\mathcal C$ for the existence of the local times of LMSS. The condition $\mathcal C$ significantly improves the sufficient conditions in \cite{MWX08} for multifractional Brownian sheets and in  \cite{shen2020local} for linear multifractional stable sheets with $0<\alpha<2$.
\begin{theorem}
\label{thm:existence}
Assume $\alpha\in(0,2]$. Let $X^{H(\bullet)}$ be an $(N,d)$-LMSS with Hurst functional index  $H(u)=\left(h_1(u),\right.$ 
$\left.\ldots,h_N(u)\right)$ and let $I=[\epsilon,T]^N$ with $0<\epsilon<T$. $X^{H(\bullet)}$ admits an $L^2(\lambda_d)$-integrable local time 
$L(\bullet,I)$ almost surely if and only if the following condition $\mathcal C$ holds:
\begin{equation}\footnotesize
\label{suf_nec_condition}
\mathcal C:\quad\quad d\le \inf_{v\in I}\sum_{l=1}^N\frac{1}{h_l(v)}~~\mbox{and}~~\int_{I}\bigg(\sum\limits_{l=1}^{N} \frac{1}{h_l(v)}-d\bigg)^{-1}\ud v<\infty.
\end{equation}
\end{theorem}
\begin{proof}  By Remark \ref{re:LT}, we first assume that $|T-\epsilon|$ is \textit{sufficiently small} so that 
Lemma \ref{lem::alpha_X_upp_low_bound} is applicable.

Denote by $\mu_I$ the occupation measure of $X^{H(\bullet)}$ on $I$ (see (\ref{occupation})) .
The Fourier transform of $\mu_I$ is
\begin{equation}\footnotesize 
\label{u_Fourier}
\hat{\mu}_I (\xi) = \int_I e^{i\langle\xi,X^{H(u)}(u)\rangle} \ud u.
\end{equation}
\normalsize
Define 
\begin{equation}\footnotesize
\label{def:J_I}
\mathcal{J}(I):=\mathbb{E} \int_{\mathbb{R}^d} |\hat{\mu}_I(\xi)|^2 \ud\xi.
\end{equation}
Plugging (\ref{u_Fourier}) into (\ref{def:J_I}) and applying Fubini's theorem, we get
\begin{equation}\footnotesize
\label{J_I}
\mathcal{J}(I) =\int_I \int_I \int_{\mathbb{R}^d}
\mathbb{E} \exp(i \langle\xi, X^{H(u)}(u) - X^{H(v)}(v)\rangle) \ud\xi \ud u \ud v.
\end{equation}
According to \cite[Theorem 21.9]{GH80}, Theorem \ref{thm:existence} is equivalent to: 
$\mathcal{J}(I)<\infty$ if and only if (\ref{suf_nec_condition}) holds. It follows from (\ref{J_I}), (\ref{characteristic}) 
and the following equation: for any constants  $a>0,b\ge0$, and $A>0$,
\begin{equation}\footnotesize
\label{int_exp_simpl}
    \int_{-\infty}^{+\infty}|x|^be^{-|x|^a A}\ud x =\frac{2}{a} \Gamma\Big(\frac{1+b}{a}\Big)A^{-(1+b)/a},
\end{equation}
that
\begin{equation}\footnotesize
\label{step_1}
\mathcal{J}(I) = 2^d\Big(\Gamma\Big(\frac{1}{\alpha}+1\Big)\Big)^d\int_I\int_I \big\| X_1^{H(u)}(u) - X_1^{H(v)}(v) \big\|_{\alpha}^{-d}\ud u\ud v.
\end{equation}
Since $|T-\epsilon|$ is sufficiently small, by Lemma \ref{lem::alpha_X_upp_low_bound}, there exist constants $c_{2,1}$, $c_{2,2}>0$ such that
 for every $u,\,v\in I$ with $u\ne v$,
\begin{equation}\footnotesize
    \label{step_2}
    \Big(c_{2,2} \sum_{l=1}^N |u_l-v_l|^{ h_l(v)} \Big)^{-d} \leq \big\| X_1^{H(u)}(u) - X_1^{H(v)}(v) \big\|_{\alpha}^{-d}
    \leq  \Big(c_{2,1} \sum_{l=1}^N |u_l - v_l|^{ h_l(v)} \Big)^{-d}.
\end{equation}
It follows from (\ref{step_1}) and (\ref{step_2}) that
\footnotesize
\begin{equation*}
c_{2,7}
 \int_I\int_I\Big(\sum_{l=1}^N |u_l-v_l|^{ h_l(v)}\Big)^{-d}\ud u\ud v \leq \mathcal J(I) \leq
c_{2,8} \int_I\int_I\Big(\sum_{l=1}^N |u_l-v_l|^{ h_l(v)}\Big)^{-d}\ud u\ud v,
\end{equation*}
\normalsize
where $
c_{2,7}
=(2\Gamma(1/\alpha+1)c_{2,2}^{-1})^d~\mbox{and} ~c_{2,8}
=(2\Gamma(1/\alpha+1)c_{2,1}^{-1})^d.
$ 
Therefore to prove the theorem it suffices to verify
\begin{equation}\footnotesize
\label{step_4}
\int_I\int_I\Big(\sum_{l=1}^N |u_l-v_l|^{ h_l(v)}\Big)^{-d}\ud v\ud u<\infty~\mbox{if and only if (\ref{suf_nec_condition}) holds.}
\end{equation}
To this end, we will prove a more general result: for any function $\theta(\bullet)$ continuous over $I$,
\begin{equation}\footnotesize
\label{step_general}
\begin{split}
\int_I\int_I\Big(\sum_{l=1}^N |u_l-v_l|^{ h_l(v)}\Big)^{-\theta(v)}\ud v\ud u<\infty \ \Longleftrightarrow \
 &\theta(v)\le \sum_{l=1}^N\frac{1}{h_l(v)}~\mbox{for}~v\in I\\
 &~~\mbox{and}~~\int_{I}\Big(\sum\limits_{l=1}^{N} \frac{1}{h_l(v)}-\theta(v)\Big)^{-1}\ud v<\infty.
 \end{split}
\end{equation}
Before we prove the aforementioned claim, let us fix some notations. For $m=1,\ldots,N$, denote by
\footnotesize
$$
I_m:=[\epsilon,T]^m,~ \quad \overline u_m:=(u_1,\ldots,u_m).
$$
\normalsize
Notice that $I_N = I$. For $\overline u_m\in I_m$, $v\in I_N$, let
\footnotesize
\begin{equation*}
A_m(\overline u_m,v):=\sum_{l=1}^m |u_l-v_l|^{ h_l(v)},~ J_{m,\theta(v)}(v):=\int_{I_m}
\left(A_m(\overline u_m,v)\right)^{-\theta(v)}\ud \overline u_m~\mbox{and}~J_m:=\int_{I}J_{m,\theta(v)}(v)\ud v.
\end{equation*}
\normalsize
Then the left-hand side integral in \eqref{step_general} is $J_N$.

We now prove the following statement by using induction: for any $m = 1, \ldots, N$,
\begin{equation}\footnotesize
\label{condition_general2}  
J_m<\infty\ \ \Longleftrightarrow \ \  \theta(v)\le \sum_{l=1}^m\frac{1}{h_l(v)}~\mbox{for}~v\in I~~\mbox{and}~~
\int_{I}\bigg(\sum\limits_{l=1}^{m} \frac{1}{h_l(v)}-\theta(v)\bigg)^{-1}\ud v<\infty.
\end{equation}
Since the integral $\int_{\{v: \theta (v)\le 0\}}J_{m,\theta(v)}(v)\ud v \le c_{2,9}< \infty$ for $m = 1, \ldots, N$, the set
$\{v: \theta (v)$ $\le 0\}$ does not affect the statement. \\

\noindent\textbf{Step 1:} Consider first the case $m=1$.\\ 
Notice that in order for $J_1 < \infty$, we necessarily have $\theta(v)\le 1/h_1(v)$ for all $v\in I$. This is because
if $\theta(v)>1/h_1(v)$ for some $v\in I$, then by the continuity of $h_1$ there exists a vector $\delta \in (0, \infty)^N$ 
with equal-valued coordinates such that $\theta(u)h_1(u)>1$ for all $u\in I\cap[v-\delta,v+\delta]$. As a result, 
\footnotesize
$$
J_1\ge \int_{I\cap[v-\delta,v+\delta]}J_{1,\theta(u)}(u)\ud u=\int_{I\cap[v-\delta,v+\delta]}\int_\epsilon^T|\varepsilon-u_1|^{-\theta(u)h_1(u)}\ud \varepsilon\ud u=\infty,
$$
\normalsize
where $u_1$ denotes the first coordinate of $u$. Hence we may assume $\theta(v)\le 1/h_1(v)$ for all $v\in I$. Subject to this constraint, we can write $I = \mathcal V \cup \mathcal V_0$,
with 
\footnotesize 
\begin{equation*}
\mathcal V := \left\{v\in I:~\theta(v)h_1(v)<1\right\}~\mbox{ and }~\mathcal V_0 := \left\{v\in I:~\theta(v)h_1(v)=1\right\}.
\end{equation*}\normalsize
Then two cases follow.
\begin{description}
\item[Case 1:] $\mathcal V$ is dense in $I$, i.e., $\overline{\mathcal V}=I$. 
\end{description}
Since the Lebesgue measures of the open sets $\mathcal V$ and  $I$ are equal, we have
\begin{equation}\footnotesize\label{dense}
\int_{I}J_{1,\theta(v)}(v)\ud v=\int_{\mathcal V}J_{1,\theta(v)}(v)\ud v~\mbox{and}~\int_{I}\Big(\frac{1}{h_1(v)}-\theta(v)\Big)^{-1}\ud v=\int_{\mathcal V}\Big(\frac{1}{h_1(v)}-\theta(v)\Big)^{-1}\ud v.
\end{equation}
For all $v\in \mathcal V$, we can write
\begin{equation}\footnotesize
\label{int_J_v}
J_{1,\theta(v)}(v)=\frac{(v_1-\epsilon)^{1-\theta(v)h_1(v)}+(T-v_1)^{1-\theta(v)h_1(v)}}{1-\theta(v)h_1(v)}.
\end{equation}
Using (\ref{int_J_v}) and the fact that $h_1(v)\in(m_1,M_1)$ for all $v\in \mathcal V$, there exist $c_{2,10},\,c_{2,11}>0$ such that 
for all $v\in \mathcal V$,
\begin{equation}\footnotesize
\label{ineq:bound_J_1d}
c_{2,10}\Big(\frac{1}{h_1(v)}-\theta(v)\Big)^{-1}\le J_{1,\theta(v)}(v)\le c_{2,11}\Big(\frac{1}{h_1(v)}-\theta(v)\Big)^{-1}.
\end{equation}
It follows from (\ref{dense}) and (\ref{ineq:bound_J_1d}) that 
\footnotesize
$$
c_{2,10}\int_{I}\Big(\frac{1}{h_1(v)}-\theta(v)\Big)^{-1}\ud v\le \int_{I}J_{1,\theta(v)}(v)\ud v\le c_{2,11}\int_{I}\Big(\frac{1}{h_1(v)}-\theta(v)\Big)^{-1}\ud v.
$$
\normalsize
Therefore in Case 1, 
\footnotesize
$$
J_1<\infty\ \ \Longleftrightarrow \ \ \theta(v)\le \frac{1}{h_1(v)}~\mbox{for}~v\in I~\mbox{and}~\int_{I}\Big(\frac{1}{h_1(v)}-\theta(v)\Big)^{-1}\ud v<\infty.
$$
\normalsize
\begin{description}
\item[Case 2:] $\mathcal V$ is not dense in $I$.
\end{description}
In this case, $\mathcal V_0$ is a closed and non-empty set. Let us first show the interior $\mathring{\mathcal V}_0\neq\emptyset$. 
If $\mathring{\mathcal V}_0=\emptyset$,  then for every $v_0\in \mathcal V_0$, there is a sequence 
$\{v_k\}_{k\ge1}$ in $\mathcal V$ such that $v_k\to v_0$ as $k\to\infty$. Therefore $v_0\in \overline{\mathcal V}$. 
This implies that $\mathcal V_0\subseteq \overline{\mathcal V}$ and $\mathcal V$ is dense in $I$, which is a contradiction. 

Now, since $\mathring{\mathcal V}_0\neq\emptyset$, there exist $v_0\in\mathcal V_0$ and $\delta \in (0, \infty)^N$ with 
equal-valued coordinates such that $I\cap [v_0-\delta,v_0+\delta]\subseteq \mathcal V_0$. Consequently,
\footnotesize$$
\int_{I}J_{1,\theta(v)}(v)\ud v\ge\int_{I\cap[v_0-\delta,v_0+\delta]}\int_\epsilon^T|\varepsilon-v_{1}|^{-\theta(v)h_1(v)}\ud \varepsilon\ud v=\infty.
$$\normalsize
Therefore $J_1 = \infty$ in Case 2.

Combining Cases 1 and 2, we obtain that $J_{1}<\infty$ implies
\begin{equation}\footnotesize
\label{finite_int_d}
 \theta(v)\le \frac{1}{h_1(v)}~\mbox{ for}~ v\in I~\mbox{and}~ \int_{I}\Big(\frac{1}{h_1(v)}-\theta(v)\Big)^{-1}\ud v<\infty
\end{equation}
and $\mathcal V$ is dense in $I$. 
This proves the necessity part. In the other direction, we see that by using similar argument in Case 2, (\ref{finite_int_d}) 
implies that $\mathcal V$ should be dense in $I$. Then it follows from Case 1 that $J_1 < \infty$.  
Therefore, we have shown that $J_{1}<\infty$ if and only if  (\ref{finite_int_d}) is satisfied.

\noindent\textbf{Step 2:} Assume that for some $n\in\{1,\ldots,N-1\}$,
\footnotesize
\begin{equation}
\label{induction_assump}
J_n=\int_{I}J_{n,\theta(v)}(v)\ud v<\infty~\Longleftrightarrow ~\theta(v)\le\sum_{l=1}^n\frac{1}{h_l(v)}~\mbox{for}~v\in I~~\mbox{and}~~ 
\int_{I}\bigg(\sum_{l=1}^n\frac{1}{h_l(v)}-\theta(v)\bigg)^{-1}\ud v<\infty.
\end{equation}
\normalsize
Now we consider $J_{n+1}$. By applying (\ref{lem::ele_int_result_2}), we have for any $v\in I$,
\begin{equation}\footnotesize
\label{J:n+1}
\begin{split}
&J_{n+1,\theta(v)}(v)=\int_{I_{n+1}}\left(A_{n+1}(\overline u_{n+1},v)\right)^{-\theta(v)}\ud \overline u_{n+1} \\
&=\int_{I_n}\bigg\{\int_\epsilon^T\big(A_n(\overline u_n,v)+|u_{n+1}-v_{n+1}|^{ h_{n+1}(v)}\big)^{-\theta(v)}\ud u_{n+1}\bigg\}\ud \overline u_n\  \\
&\asymp
	\left\{\begin{array}{llll}
	\int_{I_n} \big(A_n(\overline u_n,v)\big)^{-(\theta(v)-{1}/{h_{n+1}(v)})}\ud \overline u_{n}&\mbox{if }\theta(v)h_{n+1}(v)>1,\\
	\int_{I_n}\log\left((1+(T-v_{n+1})\left(A_n(\overline u_n,v)\right)^{-1/h_{n+1}(v)})\right.\\
\left.~~\times(1+(v_{n+1}-\epsilon)\left(A_n(\overline u_n,v)\right)^{-1/h_{n+1}(v)})\right)\ud \overline u_{n}&\mbox{if }\theta(v)h_{n+1}(v)=1,\\
	\int_{I_n}1\ud \overline u_{n}&\mbox{if }\theta(v)h_{n+1}(v)<1,\\
	\end{array}\right. \\
	&=\left\{\begin{array}{lll}
	J_{n,\theta(v)-1/h_{n+1}(v)}(v)&\mbox{if }\theta(v)h_{n+1}(v)>1,\\
	\int_{I_n}\log\left((1+(T-v_{n+1})\left(A_n(\overline u_n,v)\right)^{-1/h_{n+1}(v)})\right.\\
\left.~~\times(1+(v_{n+1}-\epsilon)\left(A_n(\overline u_n,v)\right)^{-1/h_{n+1}(v)})\right)\ud \overline u_{n}&\mbox{if }\theta(v)h_{n+1}(v)=1,\\
	(T-\epsilon)^n&\mbox{if }\theta(v)h_{n+1}(v)<1.
	\end{array}\right. 
\end{split}
\end{equation}
From  (\ref{J:n+1}) we see that 
 $
\int_{\{\theta(v)h_{n+1}(v) \le 1\}} J_{n+1,\theta(v)}(v) \ud v<\infty.
$ 
Hence $J_{n+1}<\infty$ if and only if  
\begin{equation}\footnotesize
\label{condition_1}
\int_{\{\theta(v)h_{n+1}(v) > 1\}} J_{n+1,\theta(v)}(v) \ud v<\infty.
\end{equation}
By (\ref{J:n+1}) and the remark below \eqref{condition_general2}, we see that \eqref{condition_1} is equivalent to 
\begin{equation}\footnotesize
\label{condition_1b}
\int_{I}   J_{n,\theta(v)-1/h_{n+1}(v)}(v)\ud v<\infty.
\end{equation}
Replacing $\theta(v)$ in the induction hypothesis (\ref{induction_assump}) with $\theta(v)-1/h_{n+1}(v)$ yields that, \eqref{condition_1b} holds if and only if
\footnotesize
\begin{equation*}
\theta(v)-\frac{1}{h_{n+1}(v)}\le\sum_{l=1}^n\frac{1}{h_l(v)}~\mbox{for}~v\in I~~\mbox{and}~~\int_{I}\bigg(\sum_{l=1}^{n}\frac{1}{h_l(v)}-\Big(\theta(v)-\frac{1}{h_{n+1}(v)}\Big)\bigg)^{-1}\ud v<\infty.
\end{equation*}
\normalsize
We conclude that  (\ref{condition_general2}) and thus (\ref{step_general}) are proved.  Taking $\theta(\bullet)\equiv d$ 
in (\ref{step_general}) yields (\ref{step_4}). Theorem \ref{thm:existence} is proved for $|T-\epsilon|>0$ being sufficiently small.

Finally we consider an arbitrary $I=[\epsilon,T]^N$ and let $I_1,\ldots,I_P$ be an arbitrary partition (rectangles) of $I$ such that the size of each $I_i$ is sufficiently small.  According to Remark \ref{re:LT} and the fact that
\footnotesize
$$
d\le \inf_{v\in I_i}\sum_{l=1}^N\frac{1}{h_l(v)}~~\mbox{and}~~\int_{I_i}\bigg(\sum_{l=1}^{N}\frac{1}{h_l(v)}-d\bigg)^{-1}\ud v<\infty,~\mbox{for all}~i=1,\ldots,P,
$$\normalsize
is equivalent to
\footnotesize
$$
d\le \inf_{v\in I}\sum_{l=1}^N\frac{1}{h_l(v)}~~\mbox{and}~~\int_{I}\bigg(\sum_{l=1}^{N}\frac{1}{h_l(v)}-d\bigg)^{-1}\ud v<\infty.
$$\normalsize
Hence, Theorem \ref{thm:existence} holds for arbitrary $I$. The proof is complete.
\end{proof}

\begin{remark}
\label{remark_condition}
It is easy to see that $\mathcal C$ is equivalent to either the following condition  $\mathcal C_1$ or  $\mathcal C_2$ holds:
\begin{description}
\item[$\mathcal C_1$: ]
$
d<\inf\limits_{v\in I}\sum\limits_{l=1}^{N} \frac{1}{h_l(v)}.$
\item[$\mathcal C_2$: ] $
d=\inf\limits_{v\in I}\sum\limits_{l=1}^{N} \frac{1}{h_l(v)}$ and $\int_{I}\big(\sum\limits_{l=1}^{N} \frac{1}{h_l(v)}-d\big)^{-1}\ud v<\infty$.
\end{description}
\end{remark}
The integral constraint in $\mathcal C_2$ is some requirement on the convergence rate for the function $v\mapsto \sum_{l=1}^N1/h_l(v)$ to approach its infimum on $I$. It requires the function $v\mapsto \sum_{l=1}^N1/h_l(v)$ to be \enquote{rough enough} around 
its minimizers in $I$. We can see that linear fractional stable sheets do not satisfy $\mathcal C_2$. As a result our  Theorem \ref{thm:existence} includes  \cite[Theorem 2.2]{ARX07b} as a particular case. In Table \ref{table1} below we compare our results to the literature ones in more detail. From the table we see that our Theorem \ref{thm:existence} improves the sufficient conditions  in \cite[Corollary 3.2]{MWX08} for multifractional Brownian sheets and in \cite[Theorem 3.1]{shen2020local} for linear multifractional stable sheets ($\alpha<2$).
\begin{table}[H]
    \centering
    \caption{Summary of conditions for the existence of local times of $(N,d)$-LMSS.}
    \scalebox{0.80}{
    \begin{tabular}{lllll}
    \hline
    Reference & $(N,d)$-LMSS Type& $\alpha$ & Condition Type & Condition\\
    \hline
     Xiao and Zhang (2002) \cite[Theorem 3.6]{Xiao2002}&Fractional Brownian sheets& $2$ &Sufficient & $\mathcal C_1$\\
    Ayache, Roueff and Xiao (2007) \cite[Theorem 2.2]{ARX07b}&Linear fractional stable sheets&$(0,2]$& Suff.~$\&$~nec.   &$\mathcal C_1$\\
    Meerschaert, Wu and  Xiao (2008)  \cite[Corollary 3.2]{MWX08} & Multifractional Brownian sheets & $2$& Sufficient& $\mathcal C_1$\\
    Shen, Yu and Li (2020) \cite[Theorem 3.1]{shen2020local}&Linear multifractional stable sheets&$(0,2)$&Sufficient&$d< \sum_{l=1}^N1/\sup_{v\in I}h_l(v)$\\
    Ding, Peng and Xiao (2022), Theorem \ref{thm:existence}&LMSS&$(0,2]$&Suff.~$\&$~nec. & $\mathcal C_1$ or $\mathcal C_2$\\
    \hline
    \end{tabular}}
     \label{table1}
\end{table}
Below we provide a simple example to illustrate how the conditions for the existence of local times to be derived. Consider an $(N,d)$-LMSS with
\footnotesize
$$
\alpha\in(0,2],~N=1~ \mbox{and}~
h_1(v)=\frac{1}{m}-(v-q)^k~ \mbox{for}~ v\in I=\Big[q,\frac{1}{m}\Big],
$$\normalsize
where the integer $m\ge 2$ and  the real numbers $q\ge0,~k>0$ are chosen to satisfy $h_1(1/m)>0$.

If $k\in(0,1)$, from \cite[Corollary 3.2]{MWX08} and \cite[Theorem 3.1]{shen2020local} we know that 
$
d<\inf_{v\in I}{1}/{h_1(v)}=m
$ 
is a sufficient condition for the existence of local times on $I$. As an improvement, Theorem \ref{thm:existence} yields 
that  $d\le m$ is  a sufficient and necessary condition, because in this case, either $\mathcal C_1$ or $\mathcal C_2$ 
is satisfied: we have either $d<m$ or $d=m$ with 
\footnotesize
$$
\int_I\left(\frac{1}{h_1(v)}-d\right)^{-1}\ud v=\int_{q}^{1/m}\left(\frac{1}{h_1(v)}-\frac{1}{h_1(q)}\right)^{-1}\ud v
=\frac{1}{m}\int_{q}^{1/m}\Big(\frac{1}{m(v-q)^k}-1\Big)\ud v<\infty.
$$
\normalsize
If $k\ge1$, it is easy to see that $\mathcal C_2$ can not hold, therefore by Theorem \ref{thm:existence}, the sufficient and 
necessary condition becomes $\mathcal C_1:$ $d< m$.

\section{Joint continuity of the local times}
\label{sec:local_times_continuity}
In this section we obtain that the assumption $\mathcal C_1$ in Remark \ref{remark_condition} is also 
a sufficient condition for the joint continuity of the local times of $(N,d)$-LMSS, which is significantly weaker than the ones 
in  \cite[Theorem 3.4]{MWX08} and \cite[Theorem 3.2]{shen2020local} for multifractional Brownian sheets  
and linear multifractional stable sheets, respectively. The main result is stated below.
\begin{theorem}
\label{thm::joint_con}
Assume $\alpha\in(0,2]$. Let $X^{H(\bullet)}$ be an $(N, d)$-LMSS. It has a jointly continuous local time on $I:=[\epsilon,T]^N$, 
provided $\mathcal C_1:$ $d < \inf_{v\in I}\sum_{l=1}^N {1}/{h_l(v)}$.
\end{theorem}

The proof of Theorem \ref{thm::joint_con} will be based on an multiparameter version of Kolmogorov's continuity theorem 
(cf. \cite{K02}) and estimates on the higher-order moments of the local times of $X^{H(\bullet)}$ (see Lemmas \ref{lem::E_L} and  
\ref{lem::E_L_bound}). The proofs of Lemmas \ref{lem::E_L} and  \ref{lem::E_L_bound} are technical, as they require a careful control 
of the upper and lower bounds for the weighted sum of the elements in $X_0^{H(\bullet)}$ in the $L^\alpha$-(quasi)norm. Similarly to
\cite{MWX08,X11,shen2020local}, we decompose $X_0^{H(\bullet)}$ into sum of independent multifractional sheets $Y_1$, $Y_2$, 
and $Z_l$, $l=1,\ldots,N$  (see (\ref{def:Z})) and control their bounds separately. The new idea in this paper is that, instead of using 
the property of local nondeterminism,  we extend the direct approach in  \cite{ARX07a} for linear fractional stable sheets to LMSS 
which allows us to derive more precise information on the upper bound for the moments of local times than those in 
\cite{MWX08, X11,shen2020local}.

Denote by
\begin{equation}\footnotesize
\label{def:g_j}
g_l(u_l, v_l) := c_{H(1)}^{1/N}\Big((u_l-v_l)_{+}^{h_l(u)-1/\alpha} - (-v_l)_{+}^{h_l(u)-1/\alpha}\Big).
\end{equation}
In terms of (\ref{eqn::def_lmss}) and (\ref{def:g_j}), for any $u\in\mathbb{R}_{+}^{N}$, we can write
\footnotesize$$
X_0^{H(u)}(u) = \int_{(-\infty,u]\backslash [0,u]} \prod_{l=1}^N g_l(u_l, v_l) M_{\alpha}(\ud v) + Y^{H(u)}(u),
$$\normalsize
where $
Y^{H(u)}(u): = \int_{[0,u]} \prod_{l=1}^N g_l(u_l, v_l) M_{\alpha}(\ud v).$ 
For any $u=(u_1,\ldots,u_N)\in[\epsilon,T]^N$ and $l=1,\ldots,N$, denote by $
R_l^1(u):=[0,\epsilon]~\mbox{and}~R_l^2(u):=(\epsilon,u_l]$. Hence the rectangle $[0,u]$ can be decomposed into 
the  union of disjoint sub-rectangles:
\footnotesize 
\begin{equation*}
\begin{split}
[0,u]&=\Bigg(\bigcup_{\substack{i_1,\ldots,i_N\in\{1,2\} \\ i_1+\ldots+i_N=N}} \prod_{l'=1}^N R_{l'}^{i_{l'}}(u)\Bigg)
\cup \Bigg(\bigcup_{\substack{i_1,\ldots,i_N\in\{1,2\} \\ i_1+\ldots+i_N=N+1}} \prod_{l'=1}^NR_{l'}^{i_{l'}}(u)\Bigg)
\cup \Bigg(\bigcup_{\substack{i_1,\ldots,i_N\in\{1,2\} \\ i_1+\ldots+i_N\ge N+2}} \prod_{l'=1}^NR_{l'}^{i_{l'}}(u)\Bigg)\\
&=[0,\epsilon]^N \cup \bigcup_{l=1}^N Q_l(u) \cup Q(u),
\end{split}
\end{equation*}
\normalsize
where 
$
Q_l(u):=R_1^{1}(u)\times\ldots\times R_{l-1}^{1}(u)\times R_{l}^2(u)\times R_{l+1}^1(u)\times\ldots\times R_N^1(u)
$
and $Q(u)$ is the union of $2^N-N-1$ disjoint sub-rectangles:
\footnotesize
$$
Q(u):=\bigcup\limits_{\substack{i_1,\ldots,i_N\in\{1,2\} \\ i_1+\ldots +i_N\ge N+2}}R_1^{i_1}(u)\times\ldots\times R_N^{i_N}(u).
$$
\normalsize
Thus, we can write
\begin{equation}\footnotesize
 \label{eqn::Y}
 Y^{H(u)}(u)= Y_1(u) + \sum_{l=1}^N Z_l(u) + Y_2(u),
 \end{equation}
 where
\begin{equation}\footnotesize
\label{def:Z}
\begin{split}
& Y_1(u):= \int_{[0,\epsilon]^N} g^{H(u)}(u,v) M_{\alpha}(\ud v);~~Y_2(u):= \int_{Q(u)} g^{H(u)}(u,v) M_{\alpha}(\ud v);\\
& Z_l(u):=  \int_{Q_l(u)} g^{H(u)}(u,v) M_{\alpha} (\ud v),~\mbox{ for }~l=1,\ldots,N, 
\end{split}
\end{equation}
with $g^{H(u)}(u,v)$ being defined in (\ref{eqn::g_u_v}). We claim that the random fields $Y_1$, $Y_2$, and $Z_l~(1\leq l \leq N)$ 
are independent since they are defined over disjoint sets. This together with (\ref{eqn::Y}) leads to the following result:  for 
$a_j\in\mathbb{R}$, $u^j \in I$ $(j=1,\ldots,n)$,
\begin{equation}\footnotesize
\label{lem::inequ_X_Y_Z}
\Big\| \sum_{j=1}^n a_j X_0^{H(u^j)}(u^j) \Big\|_{\alpha}^{\alpha} \geq \Big\| \sum_{j=1}^n a_j Y^{H(u^j)}(u^j) \Big\|_{\alpha}^{\alpha}
\geq \sum_{l=1}^N \Big\| \sum_{j=1}^n a_j Z_l(u^j) \Big\|_{\alpha}^{\alpha}.
\end{equation}
Thanks to 
\eqref{lem::inequ_X_Y_Z}, the random fields  $Z_l$, $l=1,\ldots,N$ play a key role in studying the 
joint continuity of local times of $X^{H(\bullet)}$.

Lemma \ref{int_exp} below is an extension of (\ref{int_exp_simpl}) to multivariate integral and it is used 
to derive the forthcoming Lemma \ref{lem::bound_increm}. Its proof is given in 
Appendix \ref{proof_Lemma_217}.
\begin{lemma}
\label{int_exp}
Let $\alpha\in(0,2]$, $n\ge1$, $b_1,\ldots,b_n\ge0$ and let the upper triangle matrix $(a_{i,j})_{i,j=1,\ldots,n}$ satisfy 
$a_{i,i}\ne 0$ for $i=1,\ldots,n$ and $a_{i,j}=0$ for $j<i$. Then the following inequality holds: 
\footnotesize
\begin{equation*}
\int_{\mathbb R^{n}}\Big(\prod_{i=1}^n|x_i|^{b_i}\Big)e^{-\sum\limits_{i=1}^{n}|\sum\limits_{j=i}^na_{i,j}x_j|^\alpha}\ud x_1
\ldots \ud x_{n}\le c_{3,1}(n)\Big(\prod_{i=1}^na_{i,i}\Big)^{-1}\prod_{\substack{i\in\{1,\ldots,n\}\\ b_i\ne 0}}\sum_{j=1}^n|u_{i,j}|^{b_i},
\end{equation*}
\normalsize
where $(u_{i,j})_{i,j=1,\ldots,n}$ is the inverse matrix of $(a_{i,j})_{i,j=1,\ldots,n}$ and
\begin{equation}\footnotesize
\label{c_n}
c_{3,1}(n)=\Big(\prod_{i=1}^n(n^{b_i-1}\vee 1)\Big)(\frac{2}{\alpha})^n\Big(\sup_{j_1,\ldots,j_n\in\{1,\ldots,n\}}
\Gamma\big(\frac{1+\sum_{i= 1}^nb_{j_i}}{\alpha}\big)\Big)^{n}.
\end{equation}
\end{lemma}
By applying Lemma \ref{int_exp}, a crucial inequality related to the weighted sum 
of $Z_l(u^j)$, $j=1,\ldots,n$ in the $L^\alpha$-(quasi)norm is obtained in (\ref{bound_sumZ})  in Lemma \ref{lem::bound_increm} below.
This inequality is essential for estimating high-order moments of the local times of LMSS in Lemma \ref{lem::E_L}. For 
multifractional Brownian sheets ($\alpha = 2$), a result similar to (\ref{bound_sumZ}) was proved in \cite[Equations (3.25) and (3.29)]{MWX08}. 
For linear fractional stable sheets with $1 <\alpha < 2$, it follows from the proof of \cite[Equation (4.37)]{X11}. A similar inequality 
was also obtained in \cite[Equations (3.33) and (3.34)]{shen2020local} for linear multifractional stable sheets. But the argument
in \cite{shen2020local} makes use of the notions of ``metric projection'' and ``orthogonality'' in \cite{X11}, which relies on the 
assumption of $1 <\alpha < 2$. Moreover, because the dependence of $C$ on $n$ in \cite[Equation (3.22)]{shen2020local} 
is not described, their Lemma 3.6 is not strong enough for proving Theorem 3.3 in \cite{shen2020local}. Our proof of the 
inequality \eqref{bound_sumZ} in Lemma \ref{lem::bound_increm} is based on an extension of the direct approach in
 \cite{ARX07a} and provides more precise information on the constant $c_{3,2}(n)$ in (\ref{def:c0}). 
 As a special case of Lemma \ref{lem::bound_increm}, we derive in Remark 
 \ref{cor_Lemma} that the constant in (\ref{bound_sumZ_simple}) is of the form $c_{3,3}^n$. This is crucial for proving the 
 local H\"older condition for the local times in Section \ref{sec:holder}. We provide the proof of Lemma 
 \ref{lem::bound_increm} in Appendix \ref{proof_Lemma}.
\begin{lemma} 
\label{lem::bound_increm}
Assume $\alpha\in(0,2]$. For any $n\ge1$, $b_1,\ldots,b_n\ge0$, $l\in \{1,\ldots, N\}$,   and $u^j\in [\epsilon, T]^N$ $(j=1,\ldots,n)$ 
with $0=u_l^0\le u_l^1 \le \ldots \le u_l^n$, we have:
\begin{equation}\footnotesize
\label{bound_sumZ}
\int_{\mathbb R^n}\Big(\prod_{i=1}^n|x_i|^{b_i}\Big)e^{-\| \sum_{j=1}^n x_j Z_l(u^j)\|_{\alpha}^{\alpha}}\ud x_1\ldots \ud x_n 
\le c_{3,2}(n)\Big(\prod_{j=1}^n\big(u_l^j - u_l^{j-1}\big)^{-h_l(u^j)(1+\sum_{i=1}^nb_i)}\Big),
\end{equation}
where $c_{3,2}(n)>0$ does not depend on $u^j$, $j=1,\ldots,n$ and its expression is given in (\ref{def:c0}).
\end{lemma}
\begin{remark}
\label{cor_Lemma}
If $b_1=\ldots=b_n=0$, we use the convention $\prod_{\emptyset}(\bullet)\equiv1$ to observe
\footnotesize
$$
\prod_{\substack{i\in\{1,\ldots,n\}\\ b_i\ne 0}}\sum_{j=1}^n|u_{i,j}|^{b_i}=1
$$
\normalsize
and $c_{3,1}(n)=((2/\alpha)\Gamma(1/\alpha))^n$ in (\ref{c_n}). Then $c_{3,2}(n)
=c_{3,3}^n$ in (\ref{def:c0}) for some $c_{3,3}>0$ independent of $n$.  Hence 
(\ref{bound_sumZ}) becomes
\begin{equation}\footnotesize
\label{bound_sumZ_simple}
\int_{\mathbb R^n}e^{-\| \sum_{j=1}^n x_j Z_l(u^j)\|_{\alpha}^{\alpha}}\ud x \le c_{3,3}^n\prod_{j=1}^n\big|u_l^j - u_l^{j-1}\big|^{-h_l(u^j)}.
\end{equation}
Note that (\ref{bound_sumZ_simple}) has been obtained in \cite{ARX07a} for LFSS 
with $\alpha\ge1$. Our result extends it to LMSS with $\alpha\in(0,2]$. We will make use of (\ref{bound_sumZ_simple})
in the proofs of Lemma \ref{lem::E_L} and  Theorem \ref{thm::holder} (see Section \ref{sec:holder}).
\end{remark}

The lemma below is also used in the proof of Lemmas \ref{lem::E_L} and \ref{lem::E_L_bound}; it can be found 
in \cite[Lemma 2.10]{MWX08} or \cite[Lemma 3.4]{AWX08}.

\begin{lemma}
\label{rho_p_q_relation}
Let $\left(\vartheta_1,\ldots,\vartheta_N\right)\in\left(0,1\right)^N$. For any $q\in[0,\sum_{l=1}^N \vartheta_l^{-1}),$ 
let $\tau\in\{1,\ldots,N\} $ be the unique integer such that
$
    \sum_{l=1}^{\tau-1}{1}/{\vartheta_l}\leq q<\sum_{l=1}^{\tau}{1}/{\vartheta_l}
$, 
with the convention that $\sum_{l=1}^0{1}/{\vartheta_l}:=0.$ There then exists a positive constant $\Delta_\tau\leq1,$ 
depending only on $\left(\vartheta_1,\ldots,\vartheta_N\right)$, such that for every $\Delta\in(0,\Delta_\tau)$, we 
can find real numbers $p_1,\ldots,p_\tau\geq1$ satisfying:
\begin{equation}\footnotesize
\label{rho_p_q_relation_result1}
    \sum_{l=1}^{\tau}\frac{1}{p_l}=1, ~~~\frac{\vartheta_lq}{p_l}<1~~~\mbox{for all}~ l=1,\ldots,\tau~~\mbox{and}~~
    (1-\Delta)\sum_{l=1}^{\tau}\frac{\vartheta_lq}{p_l}\leq \vartheta_\tau q+\tau-\sum_{l=1}^{\tau}\frac{\vartheta_\tau}{\vartheta_l}.
\end{equation}
Moreover, let
$
\alpha_\tau:=\sum_{l=1}^{\tau}{1}/{\vartheta_l}-q 
$, 
then for any $\kappa\in(0,{\alpha_\tau}/(2\tau)),$ there is $l_0\in\{1,\ldots,\tau\}$ such that
\begin{equation}\footnotesize
\label{rho_p_q_relation_result3}
    \vartheta_{l_0}\Big(\frac{q}{p_{l_0}}+2\kappa\Big)<1.
\end{equation}
\end{lemma}

We apply Remark \ref{cor_Lemma} in the proof of Lemma \ref{lem::E_L} and apply Lemma \ref{lem::bound_increm} 
in the proof of Lemma \ref{lem::E_L_bound} below. Lemma \ref{lem::E_L} improves  \cite[Lemma 3.5]{MWX08}, through 
obtaining a smaller upper bound  for the $n$th moment of the local times of LMSS under a weaker condition. This 
upper bound is useful for proving a sharp local H\"older condition on the local times (Theorem \ref{thm::holder} in 
Section \ref{sec:holder}) and for studying fractal properties of the level sets of LMSS.
\begin{lemma}
\label{lem::E_L}
Assume $\alpha\in(0,\,2]$ and $\mathcal C_1:$ $d < \inf_{v\in I}\sum_{l=1}^N {1}/{h_l(v)}$. 
Denote by $\mathcal S(N)$ the group of permutations of $\{1,\ldots,N\}$.  For each $\sigma\in \mathcal S(N)$, let
\begin{equation}\footnotesize
\label{def:sigma}
\sigma(H(v)):=\big(h_{\sigma(1)}(v),\ldots, h_{\sigma(N)}(v)\big),~\  v\in I.
\end{equation}
Also denote by
\begin{equation}\footnotesize
\label{ineqn::d_gamma}
\gamma(H(v)):=\min\bigg\{m\in\{1,\ldots,N\}:~ d < \sum_{l=1}^{m} \frac{1}{h_{l}(v)}\bigg\}.
\end{equation}
Then, for every integer $n\geq 1$,  $x\in\mathbb{R}^d$, and
$I_{a,\delta}=\prod_{l=1}^N [a_l,\, a_l + \delta]\subseteq I$ with $\delta\in(0,1]$
sufficiently small, we have
\begin{equation}\footnotesize
\label{bound_EL_1}
\mathbb{E}[L(x,I_{a,\delta})^n] \leq c_{3,4}(n)\delta^{n\overline{\beta}},
\end{equation}
where $c_{3,4}(n)$ is given in (\ref{def:c(n)}) and $\overline{\beta} = \sup\limits_{v\in I_{a,\delta},\,\sigma 
\in \mathcal S(N)}\beta(\sigma(H(v)))$ with
\begin{equation}\footnotesize
\label{def_beta_gamma_v_v}
\begin{split}
 & \beta(\sigma(H(v))):=N-\gamma(\sigma(H(v)))+h_{\gamma(\sigma(H(v)))}(v)\bigg(\sum_{l=1}^{\gamma(\sigma(H(v)))}\frac{1}{h_{\sigma(l)}(v)}-d \bigg).
\end{split}
\end{equation}
\end{lemma}
\begin{proof}
By \cite[Equation (25.5)]{GH80}, we have: for all $x,\,y\in\mathbb{R}^d$, all Borel sets $J \subseteq I$, 
and all integers $n\geq 1$,
\begin{equation}\footnotesize
\label{eqn::E_L}
\mathbb{E}[L(x,J)^n] = (2\pi)^{-nd} \int_{J^n} \int_{\mathbb{R}^{nd}} e^{-i \sum_{j=1}^n \langle v^j,x\rangle}
\mathbb{E} \Big[e^{i \sum_{j=1}^n \langle v^j, X^{H(u^j)}(u^j)\rangle}\Big] \ud\bar{v} \ud\bar{u},
\end{equation}
where $\bar{v}:=(v^1,\ldots,v^n)$, $\bar{u}:=(u^1,\ldots,u^n)$ and  $v^j:=(v_1^j,\ldots,v_d^j) \in \mathbb{R}^d$, 
$u^j:=(u_1^j,\ldots,u_N^j)$ $\in I$ for each $j=1,\ldots,n$. By (\ref{eqn::E_L}), the fact that the coordinate processes 
$X_1^{H(\bullet)}, \ldots, X_d^{H(\bullet)}$  are independent and
identically distributed, and (\ref{characteristic_LMSS}), we have
\begin{equation}\footnotesize
\label{eqn::E_L_x_T}
\mathbb{E} \big[L(x,I_{a,\delta})^n \big] \le (2\pi)^{-nd} \int_{I_{a,\delta} ^n} \prod_{k=1}^d Q_k(\overline u) \ud\overline{u},
\end{equation}
where
\footnotesize$$
Q_k(\overline u):=\int_{\mathbb{R}^n} e^{-\| \sum_{j=1}^n v_k^j X_1^{H(u^j)}(u^j) \|_{\alpha}^{\alpha}} \ud v_k,~  \
\mbox{ with \ $v_k:=(v_k^1, \ldots, v_k^n) \in \mathbb{R}^n$.}
$$\normalsize

Let $u^*\in I_{a,\delta}$ be fixed, but arbitrary. Since
$
d < \inf_{v\in I}\sum_{l=1}^N {1}/{h_l(v)}\le \inf_{v\in I_{a,\delta}}\sum_{l=1}^N {1}/{h_l(v)},
$
(\ref{ineqn::d_gamma}) guarantees that the choice of $\gamma(H(u^*))$ is unique.  Observe that, by 
(\ref{lem::inequ_X_Y_Z}) and the fact that $\gamma(H(u^*))\le N$, we have for $k\in\{1,\ldots,d\}$,
\begin{equation}\footnotesize
\label{eqn::J_k}
Q_k (\overline u) \leq \int_{\mathbb{R}^n} e^{-\sum_{l=1}^N \| \sum_{j=1}^n v_k^j Z_l(u^j) \|_{\alpha}^{\alpha}} \ud v_k   
\leq \int_{\mathbb{R}^n} e^{-\sum_{l=1}^{\gamma(H(u^*))} \| \sum_{j=1}^n v_k^j Z_l(u^j) \|_{\alpha}^{\alpha}} \ud v_k.
\end{equation}
By (\ref{eqn::J_k}) and the generalized H\"older's inequality we get
\begin{equation}\footnotesize
\label{eqn::J_k_bound}
Q_k(\overline u) \leq \int_{\mathbb{R}^n} \prod_{l=1}^{\gamma(H(u^*))} e^{ -\| \sum_{j=1}^n v_k^j Z_l(u^j) \|_{\alpha}^{\alpha}} 
\ud v_k  \leq \prod_{l=1}^{\gamma(H(u^*))} \bigg[ \int_{\mathbb{R}^n} e^{ -p_l\| \sum_{j=1}^n v_k^j Z_l(u^j) \|_{\alpha}^{\alpha}} \ud v_k \bigg]^{1/p_l},
\end{equation}
where $p_1,\ldots,p_{\gamma(H(u^*))}\geq 1$ satisfy $\sum_{l=1}^{\gamma(H(u^*))}{1}/{p_l}=1$. 
They are chosen as in Lemma \ref{rho_p_q_relation} with $q = d$ and $\theta_l = h_l(u^*)$. Applying  (\ref{bound_sumZ_simple}) 
to (\ref{eqn::J_k_bound}) yields
\begin{equation}\footnotesize
\label{eqn::J_k_bound_cont}
Q_k(\overline u)\leq c_{3,5}^n \prod_{l=1}^{\gamma(H(u^*))} \prod_{j=1}^n (u_l^{\pi_l(j)} - u_l^{\pi_l(j-1)})^{- h_l(u^{\pi_l(j)})/p_l},
\end{equation}
where  $c_{3,5}=\max\limits_{m\in \{1,\ldots,N\}}\prod_{l=1}^{m}\big(c_{3,3}p_l^{-1/\alpha}\big)^{1/p_l}$ and, for each $l\in\{1,\ldots,N\}$, $\pi_l\in\mathcal S(n)$ satisfies
 \footnotesize
\begin{equation*}
a_l \leq u_l^{\pi_l(1)} \leq \ldots \leq  u_l^{\pi_l(n)} \leq a_l + \delta~\ \mbox{ and }~\pi_l(0)=0.
\end{equation*}
\normalsize
For each $l=1,\ldots,N$, define
\begin{equation}\footnotesize
    \label{Bigpi}
    \Pi_l:=\big\{(u_l^1,\ldots,u_l^n)\in [a_l, a_l+\delta]^n:~a_l \leq u_l^{\pi_l(1)} \leq \ldots \leq  u_l^{\pi_l(n)} \leq a_l + \delta\big\}.
\end{equation}
Due to $\mathcal H_2$ 
and the fact that $|u_l^i-u_l^j|\le \delta\le 1$ for all $u_l^i,u_l^j\in[a_l,a_l+\delta]$, we have: 
for $u^{*}\in I_{a,\delta}$,
\begin{equation}\footnotesize
\label{up_bound_u_pi_by_u_star}
\begin{split}
    &\big(u_l^{\pi_l(j)} - u_l^{\pi_l(j-1)}\big)^{- h_l(u^{\pi_l(j)})}\\
    &\le \big(u_l^{\pi_l(j)} - u_l^{\pi_l(j-1)}\big)^{-| h_l(u^{*})-h_l(u^{\pi_l(j)})|} 
    \big(u_l^{\pi_l(j)} - u_l^{\pi_l(j-1)}\big)^{- h_l(u^{*})} \\
    &\leq \big(u_l^{\pi_l(j)} - u_l^{\pi_l(j-1)}\big)^{-(h_l(u^{*})+c_0(\delta))},
    \end{split}
\end{equation}
\normalsize
where 
\begin{equation}\footnotesize
\label{c_0_delta}
c_0(\delta):=c\sum_{l=1}^N\delta^{M_l}>0, 
\end{equation}
with $c>0$ being the constant in $\mathcal H_2$. 
Combining (\ref{eqn::E_L_x_T}) - (\ref{c_0_delta}),  we derive:
\begin{equation}\footnotesize
\label{eqn::E_L_x_T_bound}
\begin{split}
&\mathbb{E}[L(x,I_{a,\delta})^n]\\ 
&\le c_{3,6}^n \sum_{\pi_1,\ldots,\pi_N\in\mathcal S(n)}\int_{\Pi_1\times\ldots\times\Pi_N}\prod_{l=1}^{\gamma(H(u^*))} 
\prod_{j=1}^n (u_l^{\pi_l(j)} - u_l^{\pi_l(j-1)})^{-d h_l(u^{\pi_l(j)})/p_l}\ud \overline u   \\
&\le c_{3,6}^n \delta^{n(N-\gamma(H(u^*)))}\\
&\hspace{1cm}\times\sum\limits_{\pi_1,\ldots,\pi_{\gamma(H(u^*))}\in\mathcal S(n)} \bigg(\prod_{l=1}^{\gamma(H(u^*))}  \int_{\Pi_l} 
\prod_{j=1}^n (u_l^{\pi_l(j)} - u_l^{\pi_l(j-1)})^{-\frac{d (h_l(u^*)+c_0(\delta))}{p_l}}\ud  u_l\bigg) ,
\end{split}
\end{equation}
\normalsize
where $c_{3,6}=(2\pi)^{-d}c_{3,5}^d$. Next consider the following integral in (\ref{eqn::E_L_x_T_bound}):
\begin{equation}\footnotesize
\label{def:I_pi}
I_{\pi_l}:= \int_{\Pi_l} \prod_{j=1}^n (u_l^{\pi_l(j)} - u_l^{\pi_l(j-1)})^{-d (h_l(u^*)+c_0(\delta))/p_l}\ud u_l.
\end{equation}
Recall that, by Lemma \ref{rho_p_q_relation}, the real numbers $p_1,\ldots,p_{\gamma(H(u^*))} $ chosen in 
(\ref{eqn::J_k_bound}) also satisfy 
$dh_l(u^*)/p_l$ $<1~\mbox{ for $l=1,\ldots,\gamma(H(u^*))$}$, 
we can then choose $c_0(\delta)$ small enough such that
\begin{equation}\footnotesize
    \label{dhpless1}
    \frac{d(h_l(u^*)+c_0(\delta))}{p_l}<1~\mbox{ for all $l=1,\ldots,\gamma(H(u^*))$}. 
\end{equation}
Then apply \cite[ Lemma 2.11]{MWX08} (or  \cite[Lemma 3.6]{AWX08}) to (\ref{def:I_pi}) to obtain: there is a constant
$c_{3,7}(l,u^*)>0$ depending only on $l$ and $u^*$ (continuously) such that
\begin{equation}\footnotesize
\label{eqn::I_pi_l_bound_cont}
  I_{\pi_l}\leq c_{3,7}^n(l,u^*)(n!)^{\frac{d(h_l(u^*)+c_0(\delta))}{p_l}-1}\delta^{n\big(1-(1-\frac{1}{n})\frac{d(h_l(u^*)+c_0(\delta))}{p_l}\big)}.
\end{equation}

Combining (\ref{eqn::E_L_x_T_bound}) and (\ref{eqn::I_pi_l_bound_cont}) yields
\footnotesize
\begin{equation}
\label{eqn::E_L_x_T_bound_cont}
\mathbb{E}[L(x,I_{a,\delta})^n] \leq c_{3,8}(n) (n!)^{(1-\frac{1}{n})\sum_{l=1}^{\gamma(H(u^*))}\frac{dh_l(u^*)}{p_l}}  
\delta^{n(N-(1-\frac{1}{n})\sum_{l=1}^{\gamma(H(u^*))}\frac{dh_l(u^*)}{p_l})},
\end{equation}
\normalsize
where 
\footnotesize
\begin{equation}\footnotesize
\label{def:c3n}
c_{3,8}(n)=c_{3,6}^n2^{Nn}\sup\limits_{\substack{u\in I\\ m\in \{1,\ldots,N\}\\  \delta\in[0,1]}}
\Bigg\{\prod\limits_{l=1}^{m}c_{3,7}^n(l,u)(n!)^{\frac{dc_0(\delta)}{p_l}}
\delta^{-n\frac{dc_0(\delta)}{p_l}}\Bigg\}
\end{equation}
\normalsize 
does not depend on $\delta$. It is worth noting that the above $\sup_{\delta\in[0,1]}\delta^{-c_0(\delta)} <\infty$, thanks to the fact that
$
\lim_{\delta\to0}\delta^{-c\sum_{l=1}^N\delta^{M_l}}=1.
$
Applying Lemma \ref{rho_p_q_relation} with $\Delta=n^{-1}, q=d, \vartheta_l=h_l(u^*)$, we obtain
\begin{equation}\footnotesize
    \label{p_2}
    \begin{split}
 &\Big(1-\frac{1}{n}\Big)\sum_{l=1}^{\gamma(H(u^*))}\frac{h_l(u^*)d}{p_l}\le
 h_{\gamma(H(u^*))}(u^*)d+\gamma(H(u^*))-\sum_{l=1}^{\gamma(H(u^*))}
 \frac{h_{\gamma(H(u^*))}(u^*)}{h_{l}(u^*)}.
 \end{split}
\end{equation}
Therefore, \eqref{eqn::E_L_x_T_bound_cont} together with (\ref{p_2}) yields
\begin{equation}\footnotesize
\label{bound_EL_step1}
\mathbb{E}[L(x,I_{a,\delta})^n]\leq c_{3,8}(n) (n!)^{N-\beta(H(u^*))} \delta^{n\beta(H(u^*))}
\leq c_{3,4}(n) \delta^{n\beta(H(u^*))},
\end{equation}
where $\beta(\bullet)$ is defined in (\ref{def_beta_gamma_v_v}) and 
\begin{equation}\footnotesize
\label{def:c(n)}
    c_{3,4}(n)=c_{3,8}(n) \sup_{u\in I}\Big\{(n!)^{N-\beta(H(u))}\Big\}.
\end{equation}
In (\ref{bound_EL_step1}), since $u^*$ can be arbitrarily chosen in $I_{a,\delta}$, and the $h_l(\bullet)$'s in $H(\bullet)$
can be arbitrarily ordered,  taking the infimum over $u^*\in I_{a,\delta}$ and $\sigma\in \mathcal S(N)$ on both hand sides of
(\ref{bound_EL_step1}) leads to (\ref{bound_EL_1}). 
Therefore, Lemma \ref{lem::E_L} is proved.
\end{proof}

\begin{remark} For each fixed $n\ge1$, if we let $c_0(\delta)\le 1/n$ in Lemma \ref{lem::E_L}, we obtain 
$c_{3,8}(n)\le c_{3,9}^n$ in  (\ref{def:c3n}) for some $c_{3,9}>0$, 
thanks to Stirling's formula. As a result, Lemma \ref{lem::E_L} becomes
\footnotesize
$$
\mathbb{E}[L(x,I_{a,\delta})^n]\leq c_{3,9}^n \sup_{u\in I}\Big\{(n!)^{N-\beta(H(u))}\Big\} \delta^{n\beta(H(u^*))}.
$$
\normalsize
This observation will be used in the proof of Theorem \ref{thm::holder}.
Now we compare the moment estimates in \eqref{bound_EL_1} and (\ref{bound_EL_step1})  with those in 
\cite{AWX08,MWX08, shen2020local}. 
When either $(i)$ LMSS is reduced to LFSS or $(ii)$  $H_{\gamma(H(u^*))}(u^*)=\sup_{v\in I_{a,\delta}}H_{\gamma(H(v))}(v)$, 
we can replace $c_0(\delta) $ in (\ref{up_bound_u_pi_by_u_star}) by  $0$.  As a consequence, $c_{3,8}(n)
\le c_{3,9}^n$. Hence (\ref{bound_EL_step1}) includes \cite[Equation (3.38)]{AWX08} for fractional Brownian sheets and  
\cite[Equation (3.16)]{MWX08} for multifractional Brownian sheets (where $H_{\gamma(H(u^*))}(u^*)$ is replaced with 
$\sup_{v\in I_{a,\delta}}H_{\gamma(H(v))}(v)$) as special cases. However, a stronger condition $d<\sum_{l=1}^N1/\sup_{v\in I_{a,\delta}}h_l(v)$
than that in Lemma \ref{lem::E_L} was assumed in \cite{MWX08}. In \citep[Equation (3.22)]{shen2020local}, a result similar 
to Lemma \ref{lem::E_L} was also proved for LMSS, but the dependence of $C$ on $n$ was not described there. As a consequence, 
the estimate (3.22) in \cite{shen2020local} was not strong enough for proving the claimed Theorem 3.3 in \cite{shen2020local}. 
With (\ref{def:c(n)}) our Lemma \ref{lem::E_L}  fills this gap, which is important for proving the local H\"older condition
in Theorem \ref{thm::holder}, where for each $n\ge1$, we will take $c_0(\delta)\asymp 2^{-e^n}\le 1/n$.
\normalsize
\end{remark}

\begin{remark}
For the proof of Theorem \ref{thm::joint_con}, we will make use of the multiparameter version of Kolmogorov's continuity theorem 
(cf. \cite{K02}) and only moments of the local times of large but fixed order $n$ will be needed. It is sufficient to use 
the following simpler variant of Lemma \ref{lem::E_L}. In the last inequality of (\ref{eqn::J_k}), we replace $\gamma(H(u^*))$ with 
$N$ and apply the generalized H\"older's inequality to the second inequality in  (\ref{eqn::J_k_bound}) 
with $N$ positive numbers $p_1,\ldots, p_N$ defined by 
 $
p_l:=\sum_{l'=1}^N{h_l(u^*)}/{h_{l'}(u^*)}, ~~l=1,\ldots,N.
$ 
Then following the same proof in Lemma \ref{lem::E_L}, we have, for every interval $I_{a,\delta}
=\prod_{l=1}^N [a_l, $ $a_l + \delta] \subseteq I$, (\ref{eqn::E_L_x_T_bound_cont})  becomes
\begin{equation}\footnotesize
\label{eqn::E_L_x_T_bound_cont_spec}
\begin{split}
&\mathbb{E}[L(x,I_{a,\delta})^n]
\leq c_{3,10}(n) (n!)^{\sum_{l=1}^{N}\frac{dh_l(u^*)}{p_l}} \prod_{l=1}^{N} \delta^{n(1-\frac{dh_l(u^*)}{p_l})}
=c_{3,10}(n) (n!)^{N\nu}\big(\lambda_N(I_{a,\delta})\big)^{n(1-\nu)},
\end{split}
\end{equation}
where $\nu:=d/(\sum_{l=1}^N{1}/{h_l(u^*)})$, $c_{3,10}(n)>0$ does not depend on $\delta$.
\end{remark}

Lemma \ref{lem::E_L_bound} below is another key step leading to Theorem \ref{thm::joint_con}. We provide its proof in 
Appendix \ref{Proof_Lemma_218}.
\begin{lemma}
\label{lem::E_L_bound}
Let $X^{H(\bullet)}$ be an $(N, d)$-LMSS and assume $d < \inf_{v\in I}\sum_{l=1}^N {1}/{h_l(v)}$. Then for any 
 integer $n\geq1$, there exists $c_{3,11}(n)>0$ such that for any subintervals $I_{a,\delta}=\prod_{l=1}^N [a_l, a_l 
 + \delta]\subseteq I$ with $\delta \in(0,1]$ small enough, any $x,y\in\mathbb{R}^d$ with $|x-y|\leq 1$,
\begin{equation}\footnotesize
\label{lem::E_L_bound_result}
\begin{split}
&\mathbb{E}\big[|L(x,I_{a,\delta}) - L(y,I_{a,\delta})|^n\big] \\
&\leq c_{3,11}(n)\inf_{v\in I_{a,\delta},\sigma\in\mathcal S(N)}\Big\{|x-y|^{n\kappa_n(\sigma(H(v)))} 
\delta^{n(\beta(\sigma(H(v))-(n-1)h_{\gamma(\sigma(H(v)))}(v)\kappa_n(\sigma(H(v))))}\Big\},
\end{split}
\end{equation}
where $\sigma(H(v))$, $\gamma(H(v))$, and $\beta(H(v))$ are defined in (\ref{def:sigma}),  (\ref{ineqn::d_gamma}), and (\ref{def_beta_gamma_v_v}), respectively; for each $v\in I_{a,\delta}$, $\kappa_n(H(v))$ (depending on $n$) 
is some real number satisfying
\footnotesize
\begin{equation}
    \label{def:kappa}
n\kappa_n(H(v))\in\Big(0,1\wedge\frac{\alpha(H(v))}{2\gamma(H(v))}\Big)~\mbox{ with}~ \alpha(H(v))
:=\sum_{l=1}^{\gamma(H(v))} \frac{1}{h_l(v)} - d.
\end{equation}
\end{lemma}
Now we are ready to prove Theorem \ref{thm::joint_con}.
\begin{proof} [Proof of Theorem \ref{thm::joint_con}]
Let $I=[\epsilon,T]^N$ and $d < \inf_{v\in I}\sum_{l=1}^N {1}/{h_l(v)}$. It follows from Lemma \ref{lem::E_L_bound} 
and the multiparameter version of Kolmogorov's continuity theorem in \cite{K02} that for every $I_{a,\delta}\subseteq I$, 
the LMSS  $\{X^{H(u)}(u)\}_{u\in\mathbb{R}_{+}^N}$ has almost surely a local time $L(x, I_{a,\delta})$ that is continuous 
for all $x\in\mathbb{R}^d$.

To prove the joint continuity, observe from Lemma \ref{ineq:triagnle} that for all $x,y\in\mathbb R^d$ and $s,t\in I$ such that $|s-t|>0$ small 
enough, we have for $n\ge1$,
\begin{equation}\footnotesize
\label{joint_conti_E_L}
\begin{split}
&\mathbb{E} [|L(x, I_{a,s-a}) - L(y,I_{a,t-a})|^n]\\
& \leq2^{n-1}\left(\mathbb{E} |L(x, I_{a,s-a}) - L(x,I_{a,t-a})|^n+\mathbb{E} |L(x, I_{a,t-a}) - L(y,I_{a,t-a})|^n\right).
\end{split}
\end{equation}
The term $L(x, I_{a,s-a}) - L(x,I_{a,t-a})$ in 
(\ref{joint_conti_E_L}) can be rewritten as a sum of a finite number (which only depends on $N$) of terms of $L(x,I_j )$, 
where each $I_j$ is a closed subinterval of $I$ satisfying $\lambda_N(I_j)\le c|s-t|$, with $c>0$ not depending on $s,t$. 
Then for $|s-t|$ small enough, we apply (\ref{eqn::E_L_x_T_bound_cont_spec}) to bound it as
$
\mathbb{E}|L(x,I_j )|^n\leq c_{3,12}(j,n)\left(\lambda_N(I_{j})\right)^{n\nu_1}\leq c_{3,13}(j,n) |s-t|^{n\nu_1}$, 
where $c_{3,12}(j,n),\,c_{3,13}(j,n)>0$ are constants that do not depend on the 
edge lengths of $I_j$ and $\nu_1\in(0,1).$ Hence the first
 term in (\ref{joint_conti_E_L}) can be bounded as
 $
\mathbb{E} |L(x, I_{a,s-a}) - L(x,I_{a,t-a})|^n\leq c_{3,14}(n) |s-t|^{n\nu_1}.
$ 
On the other hand, the difference $L(x, I_{a,t-a}) - L(y,I_{a,t-a})$ in (\ref{joint_conti_E_L}) can be rewritten as 
a sum of a finite number of terms of $L(x, I_{j}) - L(y,I_{j})$, where each $I_j$ is a closed subinterval of $I_{a,t-a}$ 
satisfying $\lambda_N(I_j)$ is small enough. Then each term can be bounded by Lemma \ref{lem::E_L_bound} as
 $
\mathbb{E} |L(x, I_{j}) - L(y,I_{j})|^n\leq c_{3,15}(n) |x - y|^{n\nu_2},
$ 
where $\nu_2\in(0,1)$. Therefore, there exist $\nu\in(0,1)$ and $c_{3,16}(n)>0$ such that (\ref{joint_conti_E_L}) yields 
 $
\mathbb{E} |L(x, I_{a,s-a}) - L(y,I_{a,t-a})|^n\leq c_{3,16}(n)\left(|x - y|+|s-t|\right)^{n\nu}$.
Again by the multiparameter version of Kolmogorov's continuity theorem, the joint continuity
of the local times on $I$ holds. The proof is complete.
\end{proof}

\section{Local H\"older condition for the local times}
\label{sec:holder}

For any fixed $x\in\mathbb{R}^d$, let $L(x,\bullet)$ be the local time of the $(N,d)$-LMSS $\{X^{H(u)}(u)\}_{u\in\mathbb{R}_{+}^N}$ 
at $x$. When the local time is jointly continuous, $L(x,\bullet)$ can be extended to be a measure supported by the level set 
$\Gamma_x = \{ u\in\mathbb{R}_{+}^N: \, X^{H(u)}(u)  = x\}$. Hence,  the following theorem on the local oscillation of 
$L(x,\bullet)$ is useful for studying the fractal properties of $\Gamma_x$. See, e.g., \cite{GH80,MWX08,X97, X09}.
Compared with  \cite[Theorem 4.3]{MWX08} for multifractional Brownian sheets,  the condition of our Theorem
\ref{thm::holder} is sharper, which can be applied to derive more precise information on the Hausdorff measure of $\Gamma_x$.
A similar result for linear multifractional stable sheets was stated in \cite[Theorem 3.3]{shen2020local} and it
was claimed that it would follow from their Lemma 3.6. As we mentioned earlier, because the dependence of $C$ on $n$ in  
\cite[Equation (3.22)]{shen2020local} is not described, Lemma 3.6 in \cite{shen2020local} is not strong enough for determining
the $\log(\log(r^{-1}))$-factor in their Theorem 3.3.  

\begin{theorem}
\label{thm::holder}
Assume $\alpha\in(0,2]$. Let $I=[\epsilon,T]^N$ and  $d<\inf_{v\in I}\sum_{l=1}^N {1}/{h_l(v)}$.
There exists a constant $c_{4,1} > 0$ such that with probability 1, 
\begin{equation}\footnotesize
\label{limsup_L_over_phi_bound}
\limsup_{r \rightarrow 0} \frac{L(x,U(t,r))}{\varphi_t(r)} \leq c_{4,1},~ \ \mbox{ for $L(x,\bullet)$-almost all $t \in I$,}
\end{equation}
where $U(t,r)$ is the open ball in $I$ with center $t\in I$ and radius $r>0$, and
the scaling function $
\varphi_t(r):=r^{\beta(H(t))}\big(\log(\log(r^{-1}))\big)^{N-\beta(H(t))},~\mbox{ for }~0<r<e^{-1},
$ 
with $\beta(H(t))$ being defined in (\ref{def_beta_gamma_v_v}).
\end{theorem}
\begin{proof}
For every integer $k>0$, define the random measure $L_k(x,\bullet)$ on the Borel subset $C$ of $I$ to be
\begin{equation}\footnotesize
\label{def_L_k_x_C}
L_k(x,C):=(2\pi)^{-d}\int_C(2\pi k)^{d/2}e^{-\frac{k|X^{H(t)}(t)-x|^2}{2}}\ud t =(2\pi)^{-d}\int_C\int_{\mathbb{R}^d}
e^{-\frac{|\xi|^2}{2k}+i\big\langle\xi,X^{H(t)}(t)-x\big\rangle}\ud \xi\ud t.
\end{equation}
According to \cite[Theorem 6.4]{GH80}, the local times have a measurable modification that satisfies the occupation density formula:
for any Borel function $g(t,x)\geq 0$ on $(t,x)\in I\times \mathbb{R}^d$,
\begin{equation}\footnotesize
\label{occupation_density_formula}
\int_I g(t,X^{H(t)}(t)) \ud t = \int_{\mathbb{R}^d} \int_I g(t,x) L(x,\ud t) \ud x.
\end{equation}
Based on (\ref{occupation_density_formula}), we can obtain
\begin{equation}\footnotesize
\label{occupation_density_formula_appl}
\int_Ce^{i\big\langle\xi,X^{H(t)}(t)\big\rangle}\ud t=\int_{\mathbb{R}^d}e^{i\left\langle\xi,x\right\rangle}L(x,C)\ud x.
\end{equation}
Since the right hand-side of (\ref{occupation_density_formula_appl}) is the characteristic function of a random 
variable with density $L(x,C)$, by the inversion theorem we can derive
$
L(x,C)=(2\pi)^{-d}\int_C\int_{\mathbb{R}^d}e^{i\langle\xi,X^{H(t)}(t)-x\rangle}\ud \xi\ud t.
$ 
Now by the continuity of the mapping $y\mapsto L(y,C)$, we have 
$
L_k(x,C)\xrightarrow[k\to\infty]{a.s.}  L(x,C)$ for every Borel set $C\subseteq I$. 
Define $f_m(t):=L(x,U(t,2^{-m}))$, $m\ge1$. From the proof of Theorem \ref{thm::joint_con},
we can see that almost surely the functions $f_m$'s are continuous and bounded. Hence, by the Lebesgue's dominated
convergence theorem,  for all integers $m,n\geq1$,
\begin{equation}\footnotesize
\label{lim_int_f_m_L_k}
    \int_I\left(f_m(t)\right)^n L_k(x,\ud t)\xrightarrow[k\to\infty]{a.s.}\int_I\left(f_m(t)\right)^n L(x,\ud t).
\end{equation}

It results from (\ref{def_L_k_x_C}), (\ref{lim_int_f_m_L_k}) and the proof of Proposition 3.1 in \cite{Pitt1978}  that for
each integer $n\geq1$,
\begin{equation}\footnotesize
\label{E_int_f_m_n_L}
\begin{split}
&\mathbb{E}\int_I\left(f_m(t)\right)^n L(x,\ud t)\\
&=\frac{1}{(2\pi)^{(n+1)d}}\int_I\int_{U(s^{n+1},2^{-m})^n}\int_{\mathbb{R}^{(n+1)d}}
e^{-i\sum_{j=1}^{n+1}\langle x,u^j\rangle} \mathbb{E}e^{i\sum_{j=1}^{n+1}\langle u^j,X^{H(s^j)}(s^j)\rangle}\ud \overline{u}\ud \overline{s}\\
&\leq\frac{1}{(2\pi)^{(n+1)d}}\int_I\int_{U(s^{n+1},2^{-m})^n}\prod_{k=1}^dQ_k(\overline{s})\ud \overline{s},
\end{split}
\end{equation}
where 
 $
Q_k(\overline{s}):=\int_{\mathbb{R}^{n+1}}e^{-\|\sum_{j=1}^{n+1} u_k^jX_1^{H(s^j)}(s^j)\|_\alpha^\alpha}\ud u_k,
$ with 
$\overline{s}:=(s^1,\ldots,s^{n+1})\in\mathbb{R}^{(n+1)N}$ and $\overline{u}:=(u^1,\ldots,u^{n+1})\in\mathbb{R}^{(n+1)d}$.
In the following, we provide an upper bound
of the right-hand side of (\ref{E_int_f_m_n_L}) for sufficiently large $m$, by modifying the proof of (\ref{bound_EL_1}) in
Lemma \ref{lem::E_L}. For consistency, we use the same notations as in the proof of Lemma \ref{lem::E_L}.

For $l=1,\ldots,N$, denote $U_l(s^{n+1},2^{-m}):=[s_l^{n+1}-2^{-m},s_l^{n+1}+2^{-m}]$ as the projection of   $U(s^{n+1},2^{-m})$ onto the $l$th dimension. For each $l=1,\ldots,N$ and each permutation  $\pi_l\in\mathcal S(n+1)$, define
\begin{equation}\footnotesize
\label{def:pi_t}
  \Pi_l:=\Big\{(s_l^1,\ldots,s_l^{n+1})\in \big(U_l(s^{n+1},2^{-m})\big)^{n}\times[\epsilon,T]:~ s_l^{\pi_l(1)} \leq \ldots \leq  s_l^{\pi_l(n+1)} \Big\}\neq \emptyset,
\end{equation}
with the convention that $s_l^{\pi_l(0)}=s_l^0:=0$. 
For each $l=1,\ldots,N$, let $j_n\in\{1,\ldots,n+1\}$ be the unique integer such that  $\pi_l(j_n)=n+1$, we then define
\begin{equation}\footnotesize
\label{def:pi_+}
\begin{split}
 &\Pi_l^-:=\Big\{(s_l^{\pi_l(1)},\ldots,s_l^{\pi_l(j_n-1)})\in \big(U_l(s^{n+1},2^{-m})\big)^{j_n-1}:~ s_l^{\pi_l(1)} \leq \ldots \leq  s_l^{\pi_l(j_n-1)}\le s_l^{n+1} \Big\};\\
 &\Pi_l^+:=\Big\{(s_l^{\pi_l(j_n+1)},\ldots,s_l^{\pi_l(n+1)})\in \big(U_l(s^{n+1},2^{-m})\big)^{n-j_n+1}:~ s_l^{n+1}\le s_l^{\pi_l(j_n+1)} \leq \ldots \leq  s_l^{\pi_l(n+1)}\Big\}.
  \end{split}
\end{equation}
It results from \eqref{eqn::J_k}, \eqref{eqn::J_k_bound}, and Remark \ref{cor_Lemma} that
\begin{equation}\footnotesize
\label{Q_k_s_bar_upper_bound}
Q_k(\overline{s})\leq c_{4,2}^n\prod_{l=1}^{\gamma(H(s^{n+1}))} \prod_{j=1}^{n+1} (s_l^{\pi_l(j)} - s_l^{\pi_l(j-1)})^{- h_l(s^{\pi_l(j)})/p_l(H(s^{n+1}))},
\end{equation}
where $c_{4,2}>0$ does not depend on $n$, $\overline{s}$, $\gamma(H(s^{n+1}))$ and $p_1(H(s^{n+1}),\ldots,p_{\gamma(H(s^{n+1}))}(H(s^{n+1}))\ge1$ satisfy $\sum_{l=1}^{\gamma(H(s^{n+1}))}1/p_l(H(s^{n+1}))=1$.  
Combining (\ref{E_int_f_m_n_L}) - (\ref{Q_k_s_bar_upper_bound}),  we have
\begin{equation}\footnotesize
\begin{split}
\label{E_int_f_m_n_L_cont}
&\mathbb{E}\int_I\left[f_m(t)\right]^n L(x,\ud t)\\
&\leq c_{4,3}^n\sum_{\pi_1,\ldots,\pi_N\in\mathcal S(n)}
\int_{\Pi_1\times\ldots\times\Pi_N}\prod_{l=1}^{\gamma(H(s^{n+1}))} \prod_{j=1}^{n+1} (s_l^{\pi_l(j)} - s_l^{\pi_l(j-1)})^{-\frac{ dh_l(s^{\pi_l(j)})}{p_l(H(s^{n+1}))}}\ud \overline s\\
&\leq c_{4,3}^n\sum_{\substack{\pi_1,\ldots,\pi_{\gamma(H(s^{n+1}))}\\\in\mathcal S(n)}}\int_{[\epsilon,T]^N}(2^{n-mn})^{N-\gamma(H(s^{n+1}))}\prod_{l=1}^{\gamma(H(s^{n+1}))}\\
&\hspace{0.5cm}\times\bigg\{\int_{\Pi_l^-}
\prod_{j=1}^{j_n} (s_l^{\pi_l(j)} - s_l^{\pi_l(j-1)})^{- \frac{dh_l(s^{\pi_l(j)})}{p_l(H(s^{n+1}))}}\ud s_l^{\pi_l(1)}\ldots\ud s_l^{\pi_l(j_n-1)}\\
&\hspace{1cm}\times \int_{\Pi_l^+}
\prod_{j=j_n+1}^{n+1} (s_l^{\pi_l(j)} - s_l^{\pi_l(j-1)})^{-\frac{ dh_l(s^{\pi_l(j)})}{p_l(H(s^{n+1}))}}\ud s_l^{\pi_l(j_n+1)}\ldots\ud s_l^{\pi_l(n+1)}\bigg\} \ud s_1^{n+1}\ldots\ud s_N^{n+1},
\end{split}
\end{equation}
\normalsize
where $c_{4,3}=(c_{4,2}/(2\pi))^{d}$. Similar to (\ref{eqn::E_L_x_T_bound}) - (\ref{eqn::I_pi_l_bound_cont}),
for sufficiently large $m$ and letting $\delta^{(m)}=2^{1-m}$,  we obtain
\begin{equation}\footnotesize
\label{bound_I_pi_+}
\begin{split}
&\int_{\Pi_l^+}
\prod_{j=j_n+1}^{n+1} (s_l^{\pi_l(j)} - s_l^{\pi_l(j-1)})^{-\frac{ dh_l(s^{\pi_l(j)})}{p_l(H(s^{n+1}))}}\ud s_l^{\pi_l(j_n+1)}\ldots\ud s_l^{\pi_l(n+1)}\\
& \le c_{4,4}^{n-j_n+1}(l)((n-j_n+1)!)^{\frac{d(h_l(s^{n+1})+c_0(\delta^{(m)}))}{p_l(H(s^{n+1}))} -1}\\
&\hspace{3cm}\times(2^{-m})^{(n-j_n+1)
\big(1-(1-\frac{1}{n-j_n+1})\frac{d(h_l(s^{n+1})+c_0(\delta^{(m)}))}{p_l(H(s^{n+1}))}\big)}
\end{split}
\end{equation}
and
\begin{equation}
\label{bound_I_pi_-}
\begin{split}
&\int_{\Pi_l^-}
\prod_{j=1}^{j_n} (s_l^{\pi_l(j)} - s_l^{\pi_l(j-1)})^{-\frac{ dh_l(s^{\pi_l(j)})}{p_l(H(s^{n+1}))}}\ud s_l^{\pi_l(1)}\ldots\ud s_l^{\pi_l(j_n-1)}\\
&\le c_{4,5}^{j_n-1}(l)((j_n-1)!)^{\frac{d(h_l(s^{n+1})+c_0(\delta^{(m)}))}{p_l(H(s^{n+1}))} -1}(2^{-m})^{j_n-1\big(1-(1-\frac{1}{j_n-1})
\frac{d(h_l(s^{n+1})+c_0(\delta^{(m)}))}{p_l(H(s^{n+1}))}\big)}.
\end{split}
\end{equation}
We then use the bounds in (\ref{bound_I_pi_+}) and (\ref{bound_I_pi_-}) and the mean value theorem to obtain
\begin{equation}\footnotesize
\label{bound_I_pi}
\begin{split}
&\int_{[\epsilon,T]^N}\prod_{l=1}^{\gamma(H(s^{n+1}))}\bigg\{\int_{\Pi_l^-}
\prod_{j=1}^{j_n} (s_l^{\pi_l(j)} - s_l^{\pi_l(j-1)})^{-\frac{ dh_l(s^{\pi_l(j)})}{p_l(H(s^{n+1}))}}\ud s_l^{\pi_l(1)}\ldots\ud s_l^{\pi_l(j_n-1)}\\
&\hspace{0.5cm}\times \int_{\Pi_l^+}
\prod_{j=j_n+1}^{n+1} (s_l^{\pi_l(j)} - s_l^{\pi_l(j-1)})^{-\frac{ dh_l(s^{\pi_l(j)})}{p_l(H(s^{n+1}))}}\ud s_l^{\pi_l(j_n+1)}\ldots\ud s_l^{\pi_l(n+1)}\bigg\} \ud s_1^{n+1}\ldots\ud s_N^{n+1}\\
&\leq (T-\epsilon)^{N}\prod_{l=1}^{\gamma(H(u^*))}\Bigg\{c_{4,6}^{n}(l)(n!)^{\sum\limits_{l=1}^{\gamma(H(u^*))}\frac{d(h_l(u^*)+c_0(\delta^{(m)}))}{p_l(H(u^*))} -\gamma(H(u^*))}\\
&\hspace{2.5cm}\times(2^{-m})^{n\big(\gamma(H(u^*))-(1-\frac{1}{n})\sum\limits_{l=1}^{\gamma(H(u^*))}
\frac{d(h_l(u^*)+c_0(\delta^{(m)}))}{p_l(H(u^*))}\big)}\Bigg\},
\end{split}
\end{equation}
where $u^*\in I$ is some element depending on $m$, $n$.

We now take $n=[\log m]$, where $[\bullet]$ denotes the integer part. With this choice, the terms depending on 
$c_0(\delta^{(m)})$ in the right-hand side of (\ref{bound_I_pi}) could be upper bounded by a constant which does not depend on $m$:
\begin{equation}\footnotesize
    \label{bound_n!}
    c_{4,7}=\sup_{m\ge1}\Big\{(n!2^{mn})^{dc_0(\delta^{(m)}) }\Big\}<+\infty.
\end{equation}
The above supremum exists, thanks to the fact that
\footnotesize
$$
\lim_{m\to\infty}(n!2^{mn})^{c_0(\delta^{(m)})}=\lim_{m\to\infty}(n!2^{e^nn})^{2^{(1-e^n)}}=1.
$$
\normalsize
It follows from (\ref{E_int_f_m_n_L_cont}), (\ref{bound_I_pi}),  (\ref{bound_n!}) and the similar arguments to (\ref{eqn::I_pi_l_bound_cont}) - (\ref{p_2}) that
\begin{equation}\footnotesize
\label{E_int_f_m_n_L_bound}
    \mathbb{E}\int_I\left(f_m(t)\right)^n L(x,\ud t)\leq c_{4,8}^n (n!)^{N-\beta(H(u^*))}2^{-mn\beta(H(u^*))},
\end{equation}
where 
\footnotesize
\begin{equation*}
    c_{4,8}=\max\bigg\{2^{2N}c_{4,3}c_{4,7} (T-\epsilon)^{N} \sup_{ k\in\{1,\ldots,N\} }\Big\{\prod_{l=1}^kc_{4,6}(l)\Big\},1\bigg\}
\end{equation*}
\normalsize
does not depend on $n$. We again point out that obtaining the scaling constant $c_{4,8}^n$ in (\ref{E_int_f_m_n_L_bound}) is crucial for deriving the value of $\tau$ below. 
\normalsize

Let $\tau>0$ be a constant, the value of which will be determined later. We consider the random set
\footnotesize$$
I_m=\big\{t\in I:f_m(t)\geq \tau\varphi_{u^*}(2^{-m})\big\}.
$$\normalsize
Denote by $\mu_\omega$ the restriction of the random measure $L(x,\bullet)$ to $I$, that is, $\mu_\omega(E)
=L(x,E\cap I)$ for all Borel set $E\subseteq \mathbb{R}_+^N$. Since $n=[\log m]$, following the same approach in \citep[proof of Theorem 4.3]{MWX08} and
applying the crucial inequality (\ref{E_int_f_m_n_L_bound}) and Stirling's formula, we have
\footnotesize
\begin{equation*}
\begin{split}
\mathbb{E}\mu_\omega(I_m)&\leq \mathbb{E}\int_I L(x,\ud t)\leq\frac{\mathbb{E}\int_I \left(f_m(t)\right)^n L(x,\ud t)}
{\left(\tau\varphi_{u^*}(2^{-m})\right)^n} \leq\frac{c_{4,8}^n (n!)^{N-\beta(H(u^*))}2^{-mn\beta(H(u^*))}}
{\tau^n2^{-mn\beta(H(u^*))}(\log m)^{n(N-\beta(H(u^*)))}}\leq m^{-2},
\end{split}
\end{equation*}\normalsize
provided $\tau>0$ is chosen large enough, say, $\tau\geq c_{4,9}:=c_{4,8} e^2$. This implies
\footnotesize$$
\mathbb{E}\bigg[\sum_{m=1}^\infty\mu_\omega(I_m)\bigg]\le \sum_{m=1}^\infty m^{-2}<+\infty.
$$\normalsize
Therefore by the Borel-Cantelli lemma, with probability 1 for $\mu_\omega$-almost all $t\in I$,
\footnotesize$$
\limsup_{m\to\infty}\frac{L(x,U(t,2^{-m}))}{\varphi_{u^*}(2^{-m})}\leq c_{4,9}.
$$\normalsize
Thanks to the continuity of $h_l(\bullet)~(l=1,\ldots,N)$, it can be verified that 
\footnotesize$$
 \gamma(H(u^*))\xrightarrow[u^*\to t]{}\gamma(H(t))~~\mbox{and}~~ \beta(H(u^*))\xrightarrow[u^*\to t]{}\beta(H(t)).
$$\normalsize
Since $\gamma(H(u^*))$ and $\gamma(H(t))$ are integer-valued, we have 
$\gamma(H(u^*)) = \gamma(H(t))$ for all $m$ large enough.  By this and Condition ${\mathcal H}_2$, one can verify that there exists a constant $c_{4,10}>0$ such that $\varphi_{u^*}(2^{-m})\leq c_{4,10}\varphi_{t}(2^{-m})$ for all $m>0$.
Therefore,
\footnotesize$$
\limsup_{m\to\infty}\frac{L(x,U(t,2^{-m}))}{\varphi_{t}(2^{-m})}\leq c_{4,10} \limsup_{m\to\infty}\frac{L(x,U(t,2^{-m}))}{\varphi_{u^*}(2^{-m})}\le c_{4,11},
$$\normalsize
where $c_{4,11}=c_{4,9}c_{4,10}.$ Hence, for any $r>0$ small enough, there exists an integer $m$ such that $2^{-m}\leq r<2^{-m+1}$ and
since $\varphi_t(\bullet)$ is increasing in the neighborhood of $0$, we have
\footnotesize$$
\limsup\limits_{r\to0}\frac{L(x,U(t,{r}))}{\varphi_{t}({r})} \leq \limsup\limits_{m\to\infty}\frac{L(x,U(t,2^{-m+1}))}{\varphi_{t}(2^{-m})}
\leq c_{4,11}\sup\limits_{m\ge 1}\left\{\frac{\varphi_{t}(2^{-m+1})}{\varphi_{t}(2^{-m})}\right\}<+\infty.
$$\normalsize
This proves (\ref{limsup_L_over_phi_bound}).
\end{proof}

\begin{appendix}
\section{Proof of Lemma \ref{int_exp}}
\label{proof_Lemma_217}
For $i=1,\ldots,n$, define $
y_i:=\sum_{j=i}^na_{i,j}x_j$. Setting $a_{i,j}=0$ if $i>j$. 
Since $(a_{i,j})_{i,j=1,\ldots,n}$ is an upper triangle matrix,  we thus obtain, for $i=1,\ldots,n$,
$x_i=\sum_{j=i}^nu_{i,j}y_j$, 
for some $u_{i,j}\in\mathbb R$, $j=1,\ldots,n$, which only depend on $a_{i,j}$'s and satisfy $u_{i,i}=a_{i,i}^{-1}$, 
$u_{i,j}=0$ for $i>j$. Therefore we can write:
\begin{equation}\footnotesize
\label{step1}
\int_{\mathbb R^{n}}\Big(\prod_{i=1}^n|x_i|^{b_i}\Big)e^{-\sum\limits_{i=1}^{n}|\sum\limits_{j=i}^na_{i,j}x_j|^\alpha}\ud x
= \Big(\prod_{i=1}^na_{i,i}\Big)^{-1}\int_{\mathbb R^n}\Big(\prod_{i=1}^n\Big|\sum_{j=1}^nu_{i,j}y_j\Big|^{b_i}\Big)e^{-\sum\limits_{i=1}^n|y_i|^\alpha}\ud y,
\end{equation}
where $x=(x_1,\ldots,x_n)$ and $y=(y_1,\ldots,y_n)$. 
Using the inequality (\ref{ineq_tri}) and the multinomial formula, we obtain
\begin{equation}\footnotesize
\begin{split}
\label{step2}
&\prod_{i=1}^n\Big|\sum_{j=1}^nu_{i,j}y_j\Big|^{b_i}=\prod_{\substack{i\in\{1,\ldots,n\}\\ b_i\ne 0}}\Big|\sum_{j=1}^nu_{i,j}y_j\Big|^{b_i}\le \prod_{\substack{i\in\{1,\ldots,n\}\\ b_i\ne 0}}(n^{b_i-1}\vee 1) \sum_{j=1}^n|u_{i,j}y_j|^{b_i}\\
&=\Big(\prod_{i=1}^n(n^{b_i-1}\vee 1)\Big)\sum_{j_1,\ldots,j_n\in\{1,\ldots,n\}}\prod_{\substack{i\in\{1,\ldots,n\}\\ b_i\ne 0}}|u_{i,j_i}y_{j_i}|^{b_{i}}.
\end{split}
\end{equation}
Therefore by using (\ref{step1}), 
(\ref{step2}) and (\ref{int_exp_simpl}) iteratively,
\footnotesize
\begin{equation*}
\begin{split}
&\int_{\mathbb R^{n}}\Big(\prod_{i=1}^n|x_i|^{b_i}\Big)e^{-\sum\limits_{i=1}^{n}|\sum\limits_{j=i}^na_{i,j}x_j|^\alpha}\ud x\le \Big(\prod_{i=1}^n(n^{b_i-1}\vee 1)\Big)\Big(\prod_{i=1}^na_{i,i}\Big)^{-1}\\
&\hspace{2cm}\times\sum_{j_1,\ldots,j_n\in\{1,\ldots,n\}}\Big(\prod_{\substack{i\in\{1,\ldots,n\}\\ b_i\ne 0}}|u_{i,j_i}|^{b_{i}}\Big)\int_{\mathbb R^n}\Big(\prod_{\substack{i\in\{1,\ldots,n\}\\ b_i\ne 0}}|y_{j_i}|^{b_i}\Big)e^{-\sum_{i=1}^n|y_i|^\alpha}\ud y\\
&\le c_{3,1}(n)\Big(\prod_{i=1}^na_{i,i}\Big)^{-1}\sum_{j_1,\ldots,j_n\in\{1,\ldots,n\}}\prod_{\substack{i\in\{1,\ldots,n\}\\ b_i\ne 0}}|u_{i,j_i}|^{b_{i}}=c_{3,1}(n)\Big(\prod_{i=1}^na_{i,i}\Big)^{-1}\prod_{\substack{i\in\{1,\ldots,n\}\\ b_i\ne 0}}\sum_{j=1}^n|u_{i,j}|^{b_i},
\end{split}
\end{equation*}
\normalsize
where
\footnotesize
$$
c_{3,1}(n)=\Big(\prod_{i=1}^n(n^{b_i-1}\vee 1)\Big)(\frac{2}{\alpha})^n\Big(\sup_{j_1,\ldots,j_n\in\{1,\ldots,n\}}\Gamma\big(\frac{1+\sum_{i= 1}^nb_{j_i}}{\alpha}\big)\Big)^{n}.
$$
\normalsize
Lemma \ref{int_exp} is proved.

\section{Proof of Lemma \ref{lem::bound_increm}}
\label{proof_Lemma}
The proof is based on an extension of the direct approach in  \cite{ARX07a} for linear fractional stable sheets.  
For $n\ge 1$, by the definition of $Z_l(\bullet)$ we can write $\sum_{j=1}^nZ_l(u^j)$ as sum of independent components and obtain the following:
\begin{equation}\footnotesize
    \label{decom:Z_l}
   \Big\|\sum_{j=1}^nx_jZ_l(u^j)\Big\|_{\alpha}^\alpha=\sum_{i=1}^n\int_{Q_l(u^i)\backslash Q_l(u^{i-1})}\Big|\sum_{j=i}^nx_jg^{H(u^j)}(u^j,r)\Big|^\alpha\ud r,
\end{equation}
where $Q_l(u^0):=\emptyset$ and $Q_l(u^i)\backslash Q_l(u^{i-1})=[0,\epsilon]^{l-1}\times (u_l^{i-1},u_l^i]\times [0,\epsilon]^{N-l}$.  
Using the definition of $g_l$ in (\ref{def:g_j}), for every $i\in\{1,\ldots,n\}$,
\begin{equation}\footnotesize
    \label{decomp_ind}
    \int_{Q_l(u^i)\backslash Q_l(u^{i-1})}\Big|\sum_{j=i}^nx_jg^{H(u^j)}(u^j,r)\Big|^\alpha\ud r=c_{H(1)}^{\alpha} \int_{Q_l(u^i)\backslash Q_l(u^{i-1})}\Big|\sum_{j=i}^nx_j\prod_{p=1}^N(u_p^j-r_p)^{h_p(u^j)-1/\alpha}\Big|^\alpha\ud r,
\end{equation}
where $r=(r_1,\ldots,r_N)$. 
Applying the following change of variables to (\ref{decomp_ind}):
\footnotesize
$$
r_l\longrightarrow u_l^{i-1}+(u_l^i-u_l^{i-1})(1-r_l),
$$
\normalsize
we obtain
\begin{equation}\footnotesize
 \label{change_var}
 \int_{Q_l(u^i)\backslash Q_l(u^{i-1})}\Big|\sum_{j=i}^nx_j\prod_{p=1}^N(u_p^j-r_p)^{h_p(u^j)-1/\alpha}\Big|^\alpha\ud r=\int_{S_l(1)}|F(u^i,x,r)|^\alpha \ud r,
 \end{equation}
where
\begin{equation}\footnotesize
\label{def:S_l}
S_l(1):=\left\{r\in[0,+\infty)^N:~0\le r_p\le \epsilon~\mbox{if}~p\ne l,~0<r_l\le 1 \right\}
\end{equation}
and
 \begin{equation}\footnotesize
 \label{def:F}
F(u^i,x,r):=\sum_{j=i}^nx_j(u_l^i-u_l^{i-1})^{1/\alpha}(u_l^j-u_l^{i-1}-(u_l^i-u_l^{i-1})(1-r_l))^{h_l(u^j)-1/\alpha}\prod_{p\ne l}(u_p^j-r_p)^{h_p(u^j)-1/\alpha}.
\end{equation}
Below we distinguish with 2 cases:  $1\le \alpha\le 2$ and  $0<\alpha<1$.\\
If $\alpha\in[1,2]$, it follows from  (\ref{change_var}), H\"older's inequality and (\ref{def:F}) that 
\begin{equation}\footnotesize
    \label{sum_Z_lower_bound_alpha_1}
     \int_{S_l(1)}|F(u^i,x,r)|^\alpha \ud r\ge c_{5,1}\Big|\int_{S_l(1)}F(u^i,x,r) \ud r\Big|^\alpha= c_{5,1}\Big|\sum_{j=i}^n\theta_{i,j}x_j\Big|^\alpha,
\end{equation}
where $c_{5,1}=\epsilon^{(N-1)(1-\alpha)}$ and
\begin{equation}\footnotesize
\label{def:theta_ij}
    \theta_{i,j}=(u_l^i-u_l^{i-1})^{1/\alpha}\int_{S_l(1)}(u_l^j-u_l^{i-1}-(u_l^i-u_l^{i-1})(1-r_l))^{h_l(u^j)-1/\alpha}\prod_{p\ne l}(u_p^j-r_p)^{h_p(u^j)-1/\alpha}\ud r.
\end{equation}
Combining (\ref{decom:Z_l}), (\ref{decomp_ind}), (\ref{change_var}), and (\ref{sum_Z_lower_bound_alpha_1}) we obtain
\begin{equation}\footnotesize
\label{bound_Z_lower_1}
\Big\|\sum_{j=1}^nx_jZ_l(u^j)\Big\|_\alpha^{\alpha}\ge c_{5,1}c_{H(1)}^{\alpha}\sum_{i=1}^n\Big|\sum_{j=i}^n\theta_{i,j}x_j\Big|^\alpha.
\end{equation}
(\ref{bound_Z_lower_1}) together with Lemma \ref{int_exp} and (\ref{def:theta_ij}) yields: 
 \begin{equation}\footnotesize
 \label{bound_sum_z_0}
 \begin{split}
 &\int_{\mathbb R^n}\Big(\prod_{i=1}^n|x_i|^{b_i}\Big)e^{-\| \sum_{j=1}^n x_j Z_l(u^j)\|_{\alpha}^{\alpha}}\ud x\le \int_{\mathbb R^n}\Big(\prod_{i=1}^n|x_i|^{b_i}\Big)e^{-c_{5,1}c_{H(1)}^{\alpha}\sum_{i=1}^n|\sum_{j=i}^n\theta_{i,j}x_j|^\alpha}\ud x\\
 &\le \big(c_{5,1}c_{H(1)}^{\alpha}\big)^{-(\sum_{i=1}^nb_i+n)/\alpha}c_{3,1}(n)\Big(\prod_{i=1}^n\theta_{i,i}^{-1}\Big)\sum_{j_1,\ldots,j_n\in\{1,\ldots,n\}}\prod_{\substack{i\in\{1,\ldots,n\}\\ b_i\ne 0}}|\xi_{i,j_i}|^{b_{i}},
 \end{split}
 \end{equation}
 where $c_{3,1}(n)$ is given in (\ref{c_n}); $(\xi_{i,j})_{i,j=1,\ldots,n}$ is the inverse matrix of $(\theta_{i,j})_{i,j=1,\ldots,n}$. Note that each $\xi_{i,j}$ has the representation
 \begin{equation}\footnotesize
 \label{xi_inv}
 \xi_{i,j}=p_{i,j}((\theta_{i,j})_{i,j=1,\ldots,n})\prod_{i=1}^n\theta_{i,i}^{-1},
 \end{equation}
 where $p_{i,j}((\theta_{i,j})_{i,j=1,\ldots,n})$ denotes the $(i,j)$-element of the adjugate of the matrix $(\theta_{i,j})_{i,j=1,\ldots,n}$ thus it is a polynomial of $u^1,\ldots,u^n$. It follows from (\ref{bound_sum_z_0}) and (\ref{xi_inv}) that
 \begin{equation}\footnotesize
 \label{bound_sum_z_1}
 \begin{split}
&\int_{\mathbb R^n}\Big(\prod_{i=1}^n|x_i|^{b_i}\Big)e^{-\| \sum_{j=1}^n x_j Z_l(u^j)\|_{\alpha}^{\alpha}}\ud x\\
&\le \big(c_{5,1}c_{H(1)}^{\alpha}\big)^{-\big(\sum\limits_{i=1}^nb_i+n\big)/\alpha}c_{3,1}(n)\Big(\prod_{i=1}^n\theta_{i,i}\Big)^{-\big(1+\sum\limits_{j=1}^nb_j\big)}\hspace{-0.5cm}\sum_{j_1,\ldots,j_n\in\{1,\ldots,n\}}\prod_{\substack{i\in\{1,\ldots,n\}\\ b_i\ne 0}}|p_{i,j_i}((\theta_{i,j_i})_{i,j_i=1,\ldots,n})|^{b_{i}}\\
&\le c_{5,2}(n)\prod_{i=1}^n|u_l^i-u_l^{i-1}|^{-h_l(u^i)(1+\sum_{j=1}^nb_j)},
\end{split}
 \end{equation}
 where
\begin{equation}\footnotesize
\label{def:c1n}
\begin{split}
 &c_{5,2}(n)=\big(c_{5,1}c_{H(1)}^{\alpha}\big)^{-(\sum_{j=1}^nb_j+n)/\alpha}c_{3,1}(n)\\
 &\hspace{1cm}\times\Big(\inf_{u\in[\epsilon,T]^N}\Big\{\int_{S_l(1)}r_l^{h_l(u)-1/\alpha}\prod_{p\ne l}(u_p-r_p)^{h_p(u)-1/\alpha}\ud r\Big\}\Big)^{-(1+\sum_{j=1}^nb_j)}\\
 &\hspace{2cm}\times \sup_{u^1,\ldots,u^n\in[\epsilon,T]^N}\Bigg\{\prod_{\substack{i\in\{1,\ldots,n\}\\ b_i\ne 0}}\sum_{j=1}^n|p_{i,j}((\theta_{i,j})_{i,j=1,\ldots,n})|^{b_{i}}\Bigg\}.
 \end{split}
 \end{equation}
 \normalsize
 Next we consider the case for $0< \alpha< 1$. Since the function $\phi(r)=e^{-\beta r}$ with $\beta>0$ is convex on $[0,+\infty)$, by using Jensen's inequality we have
 \begin{equation}\footnotesize
     \label{Jensen}
     \begin{split}
         &e^{-c_{H(1)}^{\alpha}\sum_{i=1}^n\int_{S_l(1)}|F(u^i,x,r)|^\alpha \ud r}\le\frac{1}{\epsilon^{N-1}}\int_{S_l(1)}e^{-\epsilon^{N-1}c_{H(1)}^{\alpha}\sum_{i=1}^n|F(u^i,x,r)|^\alpha} \ud r.
     \end{split}
 \end{equation}
 It results from (\ref{decom:Z_l}), (\ref{change_var}),  (\ref{Jensen}) and Fubini's theorem that
 \footnotesize
  \begin{equation*}
 \begin{split}
 &\int_{\mathbb R^n}\Big(\prod_{j=1}^n|x_j|^{b_j}\Big)e^{-\| \sum_{j=1}^n x_j Z_l(u^j)\|_{\alpha}^{\alpha}}\ud x= \int_{\mathbb R^n}\Big(\prod_{j=1}^n|x_j|^{b_j}\Big)e^{-c_{H(1)}^{\alpha}\sum_{i=1}^n\int_{S_l(1)}|F(u^i,x,r)|^\alpha \ud r}\ud x\\
 &\le\frac{1}{\epsilon^{N-1}}\int_{S_l(1)}\int_{\mathbb R^n}\Big(\prod_{j=1}^n|x_j|^{b_j}\Big)e^{-\epsilon^{N-1}c_{H(1)}^{\alpha}\sum_{i=1}^n|\sum_{j=i}^n\eta_{i,j}(r)x_j|^\alpha} \ud x\ud r,
 \end{split}
 \end{equation*}
 \normalsize
 where
 \begin{equation}\footnotesize
     \label{def:eta_ij}
     \eta_{i,j}(r)=(u_l^i-u_l^{i-1})^{1/\alpha}(u_l^j-u_l^{i-1}-(u_l^i-u_l^{i-1})(1-r_l))^{h_l(u^j)-1/\alpha}\prod_{p\ne l}(u_p^j-r_p)^{h_p(u^j)-1/\alpha}.
 \end{equation}
 Then similar to the way to obtain (\ref{bound_sum_z_1}), applying again Lemma \ref{int_exp} we get
 \begin{equation}\footnotesize
 \label{bound_sum_z_2}
\int_{\mathbb R^n}\Big(\prod_{j=1}^n|x_j|^{b_j}\Big)e^{-\| \sum_{j=1}^n x_j Z_l(u^j)\|_{\alpha}^{\alpha}}\ud x\le c_{5,3}(n)\prod_{i=1}^n|u_l^i-u_l^{i-1}|^{-h_l(u^i)(1+\sum_{j=1}^nb_j)},
 \end{equation}
 where
\begin{equation}\footnotesize
\label{def:c2n}
\begin{split}
 &c_{5,3}(n)=\epsilon^{(1-N)}\big(\epsilon^{(N-1)}c_{H(1)}^{\alpha}\big)^{-(\sum_{j=1}^nb_j+n)/\alpha}c_{3,1}(n)\\
 &\hspace{1cm}\times\Big(\sup_{u\in[\epsilon,T]^N}\Big\{\int_{S_l(1)}r_l^{1/\alpha-h_l(u)}\prod_{p\ne l}(u_p-r_p)^{1/\alpha-h_p(u)}\ud r\Big\}\Big)^{1+\sum_{j=1}^nb_j}\\
 &\hspace{2cm}\times \sup_{u^1,\ldots,u^n\in[\epsilon,T]^N}\Bigg\{\int_{S_l(1)}\prod_{\substack{i\in\{1,\ldots,n\}\\ b_i\ne 0}}\sum_{j=1}^n|p_{i,j}((\eta_{i,j}(r))_{i,j=1,\ldots,n})|^{b_{i}}\ud r\Bigg\},
 \end{split}
 \end{equation}
 where each $p_{i,j}((\eta_{i,j}(r))_{i,j=1,\ldots,n})$ is the $(i,j)$-element of the adjugate of the matrix $(\eta_{i,j}(r))_{i,j=1,\ldots,n}$. 
Finally Lemma \ref{lem::bound_increm} follows from (\ref{bound_sum_z_1}) and (\ref{bound_sum_z_2}), with
\begin{equation}\footnotesize
\label{def:c0}
c_{3,2}(n) = c_{5,2}(n)\vee c_{5,3}(n).
\end{equation}

\section{Proof of Lemma \ref{lem::E_L_bound}}
\label{Proof_Lemma_218}
We first point out that, in order to show (\ref{lem::E_L_bound_result}) holds for all integer $n\ge1$, it suffices to prove that it holds for even integers $n\ge2$, thanks to the Cauchy-Schwarz inequality. Therefore in the following we assume $n$ is an even integer.

By \cite[Equation  (25.7)]{GH80}, we have: for all $x,y\in\mathbb{R}^d$, Borel sets $J \subseteq I$, 
and all even integer $n\geq 2$,
\begin{equation}\footnotesize
\label{eqn::E_L_diff}
\mathbb{E}[(L(x,J) - L(y,J))^n] = (2\pi)^{-nd} \int_{J^n} \int_{\mathbb{R}^{nd}} \prod_{j=1}^n \big(e^{-i\langle v^j,x\rangle} - e^{-i\langle v^j,y\rangle}\big)\mathbb{E}\big[ e^{i \sum_{j=1}^n \langle v^j, X^{H(u^j)}(u^j)\rangle} \big]\ud\bar{v} \ud\bar{u}.
\end{equation}
Pick any $u^*\in I_{a,\delta}$ and let $\gamma(H(u^*))\in\{1,\ldots,N\}$ be the unique integer satisfying  (\ref{ineqn::d_gamma}).
Let $\kappa_n(H(u^*))$ be the real number satisfying (\ref{def:kappa}). By the elementary inequality
\footnotesize
\begin{equation*}
|e^{ix}-1| \leq 2^{1-\kappa_n(H(u^*))} |x|^{\kappa_n(H(u^*))}, \quad \mbox{for all}~x\in \mathbb{R}
\end{equation*}\normalsize
and the triangle-type inequalities in  (\ref{ineq:triagnle}), we have for all $v^1,\ldots,v^n,x,y\in\mathbb R^d$,
\begin{equation}\footnotesize
\label{eqn::prod_e}
\prod_{j=1}^n |e^{-i\langle v^j,x\rangle} - e^{-i\langle v^j,y\rangle}| \leq 2^{n(1-\kappa_n(H(u^*)))} |x-y|^{n\kappa_n(H(u^*))} 
\sum_{\substack{j\in\{1,\ldots,n\}\\ k_j\in\{1,\ldots,d\}}} \prod_{j=1}^n |v_{k_j}^j|^{\kappa_n(H(u^*))}.
\end{equation}
\normalsize
The inequalities (\ref{lem::inequ_X_Y_Z}) and the fact that $\gamma(H(u^*))\le N$ yield
\begin{equation}\footnotesize
    \label{bound_exp_Z}
    \begin{split}
\Big\| \sum_{j=1}^n v_k^j X_k^{H(u^j)}(u^j)
\Big\|_{\alpha}^{\alpha}\ge \sum_{l=1}^{N} \Big\| \sum_{j=1}^n v_k^j Z_l(u^j)
\Big\|_{\alpha}^{\alpha}\ge \sum_{l=1}^{\gamma(H(u^*))} \Big\| \sum_{j=1}^n v_k^j Z_l(u^j)
\Big\|_{\alpha}^{\alpha}.
\end{split}
\end{equation}
Since $n\ge2$ is even, the left-hand side of (\ref{eqn::E_L_diff}) is nonnegative. Combining (\ref{eqn::E_L_diff}), (\ref{eqn::prod_e}) and (\ref{bound_exp_Z}), and using the independence of 
$X_k^{H(\bullet)}$, $k=1,\ldots,d$, we have
\begin{equation}\footnotesize
\label{eqn::E_L_x_y_bound}
\begin{split}
&\mathbb{E} [(L(x, I_{a,\delta}) - L(y,I_{a,\delta}))^n] \leq \left(2\pi\right)^{-nd}2^{n\left(1-\kappa_n\left(H\left(u^*\right)\right)\right)} |x - y|^{n\kappa_n(H(u^*))}\\ 
&\hspace{1cm}\times\sum_{\substack{j\in\{1,\ldots,n\}\\ k_j\in\{1,\ldots,d\}}}\int_{I_{a,\delta}^n} \ud\overline{u} 
\int_{\mathbb{R}^{nd}} \Big(\prod_{j=1}^n |v_{k_j}^j|^{\kappa_n(H(u^*))}\Big)\prod_{k=1}^d e^{ - \sum\limits_{l=1}^{\gamma(H(u^*))} 
\| \sum\limits_{j=1}^n v_k^j Z_l(u^j)
\|_{\alpha}^{\alpha} } \ud\overline{v}  \\
&\leq |x - y|^{n\kappa_n(H(u^*))}\hspace{-0.2cm} \sum_{\substack{j\in\{1,\ldots,n\}\\ k_j\in\{1,\ldots,d\}}}
\int_{I_{a,\delta}^n} \ud\overline{u} \prod_{k=1}^d\int_{\mathbb{R}^{n}} \Big(\prod_{j=1}^n |v_{k}^j|^{\kappa_n(H(u^*))
\eta_k(k_j)}\Big) e^{ - \sum\limits_{l=1}^{\gamma(H(u^*))} \| \sum\limits_{j=1}^n v_k^j Z_l(u^j) \|_{\alpha}^{\alpha} } \ud v_k,
\end{split}
\end{equation}
where $
\eta_k(u) := \begin{cases}
      1 & \textrm{ if } u=k, \\
      0 & \textrm{ if } u\neq k.
   \end{cases}$. Now take $\Delta=n^{-1}, q=d, \vartheta_l=h_l(u^*)$ for $l=1,\ldots,N$ in Lemma \ref{rho_p_q_relation} and let 
$p_1,\ldots,$ $p_{\gamma\left(H\left(u^*\right)\right)}$  satisfy (\ref{rho_p_q_relation_result1}). Observe that since $n\kappa_n(H(u^*))\in(0,\frac{\alpha(H(u^*))}{2\gamma(H(u^*))})$, 
it follows from (\ref{rho_p_q_relation_result3}) that there exists $l_0\in\{1,\ldots,\gamma(H(u^*))\}$
 (depending on $\kappa_n(H(u^*))$) such that
\begin{equation}\footnotesize
\label{dhpkappaless1}
    h_{l_0}(u^*)\Big(\frac{d}{p_{l_0}}+2n\kappa_n(H(u^*))\Big)<1.
\end{equation}
Combining (\ref{eqn::E_L_x_y_bound}) with the generalized H\"older's inequality, we obtain
\begin{equation}\footnotesize
\label{eqn::E_L_x_y_bound_cont}
\begin{split}
&\mathbb{E} [(L(x, I_{a,\delta}) - L(y,I_{a,\delta}))^n] \leq |x - y|^{n\kappa_n(H(u^*))}  \\
& \qquad \times \sum_{\substack{j\in\{1,\ldots,n\}\\ k_j\in\{1,\ldots,d\}}} \int_{I_{a,\delta}^n} \prod_{k=1}^d
\bigg\{\mathcal M_{l_0,k,k_1,\ldots,k_n}(\overline u)^{1/p_{l_0}}\hspace{-5mm}\prod_{\substack{ l\in\{1,\ldots,\gamma(H(u^*))\}\\
l\neq l_0} }\mathcal M_{l,k}(\overline u)^{1/p_{l}}\bigg\}\ud\overline{u},
\end{split}
\end{equation}
where
\begin{equation}\footnotesize
\label{def:M_1}
\mathcal{M}_{l_0,k, k_1,\ldots,k_n}(\overline u) :=  \int_{\mathbb{R}^{n}} \Big(\prod_{j=1}^n |v_{k}^j|^{\kappa_n(H(u^*))\eta_k(k_j)p_{l_0}}\Big)
e^{ -p_{l_0}  \| \sum_{j=1}^n v_k^j Z_{l_0}(u^j) \|_{\alpha}^{\alpha} } \ud v_k
\end{equation}
and 
\footnotesize
$$
\mathcal{M}_{l,k}(\overline u) := \int_{\mathbb{R}^{n}} e^{ -p_{l}  \| \sum_{j=1}^n v_k^j Z_l(u^j)\|_{\alpha}^{\alpha} }\ud v_k.
$$
\normalsize
Next we provide upper bounds of $\mathcal{M}_{l_0,k, k_1,\ldots,k_n}(\overline u)$ and $\mathcal{M}_{l,k}(\overline u)$, respectively.\\ \\
\noindent\textbf{Upper bound of $\mathcal{M}_{l_0,k, k_1,\ldots,k_n}(\overline u)$:}\\
Taking $b_i=\kappa_n(H(u^*))\eta_k(k_i)p_{l_0}$, $j=1,\ldots,n$ in Lemma \ref{lem::bound_increm} and using    (\ref{up_bound_u_pi_by_u_star}), 
we derive that, for $\delta>0$ small enough, there is a constant $c_{5,4}(l_0,n)>0$ such that 
\begin{equation}\footnotesize
\label{M_l_0_bound_cont}
\mathcal{M}_{l_0,k,k_1,\ldots,k_n}(\overline u)\leq c_{5,4}(l_0,n)\prod_{j=1}^n\big(u_{l_0}^{\pi_{l_0}(j)} - u_{l_0}^{\pi_{l_0}(j-1)}\big)^{-(h_{l_0}(u^*)+c_0(\delta)) 
(1+\kappa_n(H(u^*))p_{l_0}\sum\limits_{i=1}^n\eta_k(k_i) )},
\end{equation}
where $c_0(\delta)$ is given in (\ref{c_0_delta}). \\ \\
\noindent \textbf{Upper bound of $\mathcal{M}_{l,k}(\overline u)$:}\\
Similarly, applying Remark \ref{cor_Lemma} and 
(\ref{up_bound_u_pi_by_u_star}), we easily obtain, for each $l\in \{1,\ldots,\gamma(H(u^*))\}\backslash\{l_0\}$, 
there is $c_{5,5}(l)>0$ such that
\begin{equation}\footnotesize
\label{M_l_bound}
\mathcal{M}_{l,k}(\overline u) \leq c_{5,5}^n(l) \prod_{j=1}^n\big(u_{l}^{\pi_{l}(j)} - u_{l}^{\pi_{l}(j-1)}\big)^{-(h_{l}(u^*)+
c_0(\delta))}.
\end{equation}

For $l=1,\ldots,N$, let $\pi_l$ be defined as in (\ref{Bigpi}). Now combining (\ref{eqn::E_L_x_y_bound_cont}), 
(\ref{M_l_0_bound_cont}), (\ref{M_l_bound}) and using the fact that $\sum_{k=1}^d \sum_{i=1}^n\eta_k(k_i)= n$ 
for $k_i=1,\ldots,d$, we obtain
\begin{equation}\footnotesize
\begin{split}
\label{eqn::E_L_x_y_bound_cont_delta}
&\mathbb{E} [(L(x, I_{a,\delta}) - L(y,I_{a,\delta}))^n]  \leq c_{5,6}(n)|x - y|^{n\kappa_n(H(u^*))} \\
&\times \sum_{\substack{j\in\{1,\ldots,n\}\\ k_j\in\{1,\ldots,d\}}}\sum_{\substack{\pi_1,\ldots,\pi_{\gamma(H(u^*))}\\ \in\mathcal S(n)}} 
\Big(\int_{\Pi_{l_0}}\prod_{j=1}^n
\big(u_{l_0}^{\pi_{l_0}(j)} - u_{l_0}^{\pi_{l_0}(j-1)}\big)^{-(h_{l_0}(u^*)+c_0(\delta))(d/p_{l_0}+n\kappa_n(H(u^*))}\ud u_{l_0} \Big) \\
&\times\Bigg\{\prod_{\substack{l=1,\ldots,\gamma(H(u^*)) \\ l\neq l_0} } \Big(\int_{\Pi_{l}} 
\prod_{j=1}^n\big(u_{l}^{\pi_{l}(j)} - u_{l}^{\pi_{l}(j-1)}\big)^{-(h_{l}(u^*)+c_0(\delta))d/p_{l}}
\ud u_{l} \Big)\Bigg\} \delta^{n(N-\gamma((u^*)))},
\end{split}
\end{equation}
where 
\footnotesize
$$
c_{5,6}(n)=\sup_{l_0,m\in\{1,\ldots,N\}}c_{5,4}(l_0,n)^{d/p_{l_0}}\prod\limits_{{l=1,\ldots,m,~ l\ne l_0}}c_{5,5}(l)^{nd/p_{l}}.
$$
\normalsize
Since (\ref{dhpkappaless1}) holds, we are able to choose $\delta\in(0,1]$ small enough so that the following inequality also holds:
\begin{equation}\footnotesize
\label{dhpkappadeltaless1}
    \left(h_{l_0}(u^*)+c_0(\delta)\right)\Big(\frac{d}{p_{l_0}}+2n\kappa_n(H(u^*))\Big)<1.
\end{equation}
Thanks to (\ref{dhpless1}) and (\ref{dhpkappadeltaless1}), the integrals in (\ref{eqn::E_L_x_y_bound_cont_delta}) are finite. 
Then similar to the derivation of (\ref{eqn::I_pi_l_bound_cont}), 
\begin{equation}\footnotesize
\label{eqn::E_L_x_y_bound_cont_cube}
\begin{split}
&\mathbb{E} [(L(x, I_{a,\delta}) - L(y,I_{a,\delta}))^n]  \\
&\leq c_{3,11}(n)|x - y|^{n\kappa_n(H(u^*))}\delta^{n\big(N-(1-1/n)\sum_{l=1}^{\gamma(H(u^*))}dh_l(u^*)/p_l-(1-1/n)h_{l_0}(u^*)n\kappa_n(H(u^*))\big)}, 
\end{split}
\end{equation}
where
\footnotesize
\begin{equation*}
    \begin{split}
      &  c_{3,11}(n)=c_{5,6}(n)d^n\sup_{\substack{m\in \{1,\ldots,N\}\\ u\in I\\ \delta\in[0,1]}}\Big\{(n!)^{(h_{m}(u)+c_0(\delta))n\kappa_n(H(u))}\\
      &\hspace{1cm}\times\prod_{l=1}^{m}c_{5,7}^n(l,u)(n!)^{d(h_l(u)+c_0(\delta))/p_l}\delta^{-(n-1)\left(dc_0(\delta)/p_l 
      +c_0(\delta)n\kappa_n(H(u))\right)}\Big\}.
    \end{split}
\end{equation*}
\normalsize
Applying Lemma \ref{rho_p_q_relation} with $\Delta=n^{-1}, q=d, \vartheta_l=h_l(u^*)$, we obtain
\footnotesize
\begin{equation*}
\begin{split}
 \Big(1-\frac{1}{n}\Big)\sum_{l=1}^{\gamma(H(u^*))}\frac{h_l(u^*)d}{p_l}\le h_{\gamma(H(u^*))}(u^*)d+\gamma(H(u^*))-
 \sum_{l=1}^{\gamma(H(u^*))}\frac{h_{\gamma(H(u^*))}(u^*)}{h_{l}(u^*)}.
 \end{split}
\end{equation*}
\normalsize
W.l.o.g., we can assume that $0<h_{1}(u^*)\leq\ldots\leq h_{N}(u^*)<1$. Therefore,  (\ref{eqn::E_L_x_y_bound_cont_cube}) yields
\begin{equation}\footnotesize
\label{E_L_x_y_bound_done}
\mathbb{E} [(L(x, I_{a,\delta}) - L(y,I_{a,\delta}))^n] \leq c_{3,11}(n) |x - y|^{n\kappa_n(H(u^*))}\delta^{n\big(\beta(H(u^*))-(n-1)h_{\gamma(H(u^*))}(u^*)\kappa_n(H(u^*))\big)}.
\end{equation}
Since the choice of $u^*$ in (\ref{E_L_x_y_bound_done}) is arbitrary in $I_{a,\delta}$ and the order of coordinates in $H(\bullet)$ 
can be arbitrary, taking the infimum over $u^*\in I_{a,\delta}$ and $\sigma\in\mathcal S(N)$ on both hand sides of 
(\ref{E_L_x_y_bound_done}) leads to (\ref{lem::E_L_bound_result}). Lemma \ref{lem::E_L_bound} is proved.

\end{appendix}
\section*{Acknowledgements}
Yimin Xiao's research is supported in part by the NSF grant DMS-1855185. We would like to thank the referees for their thoughtful review of an earlier version of the manuscript, leading to this much improved final version. We also thank Zhiye Lu for valuable discussions related to this paper.


\bibliographystyle{imsart-number.bst}
\bibliography{ml}

\end{document}